\documentclass[11pt, reqno]{amsart}
\usepackage{latexsym, url}
\usepackage{amsthm}
\usepackage{amsmath}
\usepackage{amsfonts}
\usepackage{amssymb}
\usepackage[dvips]{graphicx}
\usepackage{xypic}
\usepackage{xcolor}
\usepackage{tikz}
\usetikzlibrary{arrows}
\usetikzlibrary{calc}
\usepackage{hyperref}
\input xy
\xyoption{all}

\newtheorem{theo}{Theorem}[section]
\newtheorem{lemma}[theo]{Lemma}

\newtheorem{propo}[theo]{Proposition}
\newtheorem{coro}[theo]{Corollary}
\newtheorem*{theo:char}{Theorem~\ref{char-single}}
\newtheorem*{theo:birk}{Theorem~\ref{birkhoff}}
\newtheorem*{theo:monad-equational}{Proposition~\ref{monad->equations}}
\newtheorem*{theo:multi}{Theorem~\ref{multi-char}}

\theoremstyle{definition}
\newtheorem{defi}[theo]{Definition}

\newtheorem{rem}[theo]{Remark}

\newtheorem{exam}[theo]{Example}
\newtheorem{exams}[theo]{Examples}

\newcommand\Inj{\operatorname{\it Inj}}

\newcommand\Mod{\operatorname{\bf Mod}}

\newcommand\Surj{\operatorname{\it Surj}}

\newcommand\op{\operatorname{op}}
\newcommand\id{\operatorname{id}}

\newcommand\Set{\operatorname{\bf Set}}
\newcommand\Cat{\operatorname{\bf Cat}}
\newcommand\Met{\operatorname{\bf Met}}

\newcommand\Str{\operatorname{\bf Str}}
\newcommand\CMet{\operatorname{\bf CMet}}
\newcommand\Ab{\operatorname{\bf Ab}}
\newcommand\Ban{\operatorname{\bf Ban}}
\newcommand\SSet{\operatorname{\bf SSet}}

\newcommand\Gra{\operatorname{\bf Gra}}
\newcommand\Pos{\operatorname{\bf Pos}}
\newcommand\MGra{\operatorname{\bf MGra}}

\newcommand\CPO{\operatorname{\bf CPO}}

\newcommand\eps{\varepsilon}

\newcommand\ca{\mathcal {A}}
\newcommand\cb{\mathcal {B}}
\newcommand\cc{\mathcal {C}}
\newcommand\cd{\mathcal {D}}

\newcommand\cg{\mathcal {G}}
\newcommand\ch{\mathcal {H}}

\newcommand\ce{\mathcal {E}}

\newcommand\ck{\mathcal {K}}
\newcommand\cl{\mathcal {L}}
\newcommand\cm{\mathcal {M}}

\newcommand\ct{\mathcal {T}}

\newcommand\cv{\mathcal {V}}

\newcommand{\tx}{\textnormal}
\newcommand{\bo}{\mathbf}

\date{November 25, 2025}
 
\begin{document}
\title[Towards enriched universal algebra]
{Towards enriched universal algebra}
\author[J. Rosick\'{y} and G. Tendas]
{J. Rosick\'{y} and G. Tendas}
\thanks{Both authors acknowledge the support of the Grant Agency of the Czech Republic under the grant 22-02964S. The second author also acknowledges the support of the EPSRC postdoctoral fellowship EP/X027139/1. We also thank Reuben Hillyard and the anonymous referee for valuable feedback} 
\address{
\newline J. Rosick\'{y}\newline
Department of Mathematics and Statistics, 
Masaryk University, Faculty of Sciences\newline
Kotl\'{a}\v{r}sk\'{a} 2, 611 37 Brno, Czech Republic\newline
\textnormal{rosicky@math.muni.cz}\newline
\newline G. Tendas\newline
Department of Mathematics and Statistics\newline
Masaryk University, Faculty of Sciences\newline
Kotl\'{a}\v{r}sk\'{a} 2, 611 37 Brno, Czech Republic\vspace{5pt}\newline
\textit{Secondary address:}\newline
Department of Mathematics, University of Manchester,\newline 
Faculty of Science and Engineering, \newline
Alan Turing Building, M13 9PL Manchester, UK\newline
\textnormal{giacomo.tendas@manchester.ac.uk}
}

\begin{abstract}
	Following the classical approach of Birkhoff, we suggest an enriched version of {\em universal algebra}. Given a suitable base of enrichment $\cv$, we define a {\em language} $\mathbb L$ to be a collection of $(X,Y)$-ary function symbols whose {\em arities} are taken among the objects of $\cv$. The class of {\em $\mathbb L$-terms} is constructed recursively from the symbols of $\mathbb L$, the morphisms in $\cv$, and by incorporating the monoidal structure of $\cv$. Then, {\em $\mathbb L$-structures} and interpretations of terms are defined, leading to {\em enriched equational theories}. In this framework we characterize algebras for finitary monads on $\cv$ as models of enriched equational theories.
\end{abstract} 
\keywords{Enriched categories, Universal algebra, Monads, Birkhoff variety}
\subjclass{18D20, 03C05, 18C05, 18C15}

\maketitle

\setcounter{tocdepth}{1}
\tableofcontents

\section{Introduction}
Universal algebra, created by Birkhoff~\cite{Bi}, deals with sets $A$ equipped with functions $f_A\colon A^n\to A$ where $f$ is a function symbol and $n$ a finite cardinal called the {\em arity} of $f$. Function symbols
together with their arities form a set $\mathbb L$ called a {\em language} (or a signature). Starting from such a language one builds {\em terms} and {\em equations} and characterizes classes of algebras satisfying certain equations as classes closed under products, substructures and quotients (Birkhoff's theorem). 

A categorical treatment of universal algebra was given by Lawvere~\cite{La} using his concept of an {\em algebraic theory}. This is the data of a category whose objects $J_n$ are indexed by finite cardinals and satisfy $J_n=(J_1)^n$. From the universal algebra point of view, morphisms $J_n\to J_m$ correspond to $m$-tuples of $n$-ary terms. Alternatively, these morphisms can be viewed as $(n,m)$-ary terms, where $n$ and $m$ are respectively the input and output arity, and the traditional superposition of terms $s(t_1,\dots,t_m)$ can be replaced with the composition $s\circ t$ where $t$ is the $(n,m)$-ary term induced by the family $t_i$, for $i\leq n$. This was further developed by Linton~\cite{L} who showed that infinitary Lawvere theories correspond to infinitary monads on $\Set$. Later, also Linton~\cite{L1}, proved that the language of $(X,Y)$-ary operations can describe monads on an arbitrary category $\cv$ if we take $X$ and $Y$ to be objects of $\cv$. The first attempt to create a syntactic concept of a term in this framework was made by the first author in~\cite{R0}.

An $(X,Y)$-ary function symbol of~\cite{L1} is interpreted, on a structure $A$, as a function $\cv(X,A)\to\cv(Y,A)$ between the homsets of $\cv$. When $\cv$ is symmetric monoidal closed, another natural but different way to interpret an $(X,Y)$-ary function symbol is as a morphism 
$$f_A\colon A^X\longrightarrow A^Y$$ 
in $\cv$, where $A^{(-)}:=[-,A]$ is the internal hom in $\cv$. Taking such a monoidal closed $\cv$ to be the base of enrichment, this leads towards a notion of {\em enriched universal algebra} which, we shall see, is captured by enriched theories and monads. 

The first attempts to look at an enriched categorical version of universal algebra, following the path of Lawvere, is due to Dubuc~\cite{Du} and Borceux and Day~\cite{BD}. Later, Power formalized the notion of enriched Lawvere theory, see~\cite{Po} and then~\cite{NP} with Nishiwaza, and the first author and Lack further generalized the concepts in~\cite{LR11}. An even more general treatment was given by Bourke and Garner~\cite{BG} who introduced the notion of {\em pretheory}; see also~\cite{LP0,Ark} for other more recent approaches. 

The pretheories of~\cite{BG} are identity-on-objects $\cv$-functors $J\colon \ca^{\op}\to\ct$ where $\ca$ is a full subcategory of $\cv$ consisting of arities. A special case is a Lawvere theory which is an identity-on-objects functor $J\colon \ca^{\op}\to\ct$ where $\ca$ is the full subcategory of $\Set$ consisting of the finite cardinals. In this case $J$ preserves finite products, making it a {\em theory} in the sense of \cite{BG}. More generally, if $\cv$ is locally finitely presentable as a closed category, then finitary enriched monads on $\ck$ correspond to enriched Lawvere theories; that is, to identity-on-objects enriched functors $J\colon \cv_f^{\op}\to\ct$ preserving finite powers, see Power~\cite{Po}.

All the enriched generalizations mentioned above follow the purely categorical approach of Lawvere, but do not provide a direct generalization of universal algebra as introduced by Birkhoff. In fact, enriched instances of classical universal algebra, with function symbols, recursively generated terms, and equations, have been developed only in specific situations: notably over posets (\cite{Bl,ADV,AFMS}), metric spaces (\cite{MPP1,MPP,A,ADV1}), and complete partial orders (\cite{ANR,ADV}). 
In this paper, we unify this fragmented picture under the same general theory, with the aim of providing new useful tools that will allow the development of universal algebra in new areas of enriched category theory.

Alternative approaches, making use of certain terms and equations, have been considered by Fiore and Hur~\cite{FH} and Lucyshyn-Wright and Parker~\cite{LP}; but these do not follow the classical approach of universal algebra where terms are recursively generated by the function symbols under change of variables and superposition.

{\bf Contents of the paper.} We begin with a language $\mathbb L$ given by a set of $(X,Y)$-ary function symbols, whose arities $X$ and $Y$ are objects of the base of enrichment $\cv$. These kind of languages were introduced in~\cite[5.1]{LP} as {\em free-form signatures}; our input arity $X$ is called an arity there and our output arity $Y$ is a parameter.
We then define {\em enriched terms} recursively as follows (Definition~\ref{terms}):
\begin{enumerate}
	\item every morphism $f\colon Y\to X$ of $\cv$ is an $(X,Y)$-ary term; 
	\item every function symbol $f:(X,Y)$ of $\mathbb L$ is an $(X,Y)$-ary term;
	\item if $t$ is an $(X,Y)$-ary term and $Z$ is an arity, then $t^Z$ is a $(Z\otimes X,Z\otimes Y)$-ary term;
	\item given $t_J=(t_j)_{j\in J}$, where $t_j$ is an $(X_j,Y_j)$-ary term, and $s$ an $(\sum_{j\in J} Y_j,W)$-ary term; then $s(t_J)$ is a $(\sum_{j\in J} X_j, W)$-ary term.
\end{enumerate}
The rules (2) and (4) are the usual starting point for terms in universal algebra expressing the fact that function symbols are terms and that we are allowed to take superposition. Rule (1) expresses variable declaration and change of variables within $\cv$. Finally, (3) witnesses another enriched aspect of our terms: the presence of a {\em power term} $t^Z$ captures the monoidal structure of $\cv$ as part of the syntactic rules defining terms. When $Z=\textstyle\sum_SI$ is a coproduct of the unit (as it happens when $\cv=\Set$), the power term $t^Z$ corresponds simply to taking the $S$-tuple $(t,\cdots,t)$.

If the base category $\cv$ is locally $\lambda$-presentable as a closed category~\cite{Kel82}, then we can talk about $\lambda$-ary languages and $\lambda$-ary terms just by restricting the arities $(X,Y)$ to be objects of $\cv_\lambda$, the full subcategory of $\cv$ spanned by the $\lambda$-presentable objects. We then define interpretation of terms in $\mathbb L$-structures:
$$ (X,Y)\tx{-ary term } t,\ \mathbb L\tx{-structure } A\ \ \mapsto\ \ t_A\colon A^X\to A^Y\ \tx{in } \cv. $$
{\em An equational $\mathbb L$-theory} $\mathbb E$ is defined as a family of equations $\{(s_j=t_j)\}_{j\in J}$ between terms of the same arity; its models are $\mathbb L$-structures satisfying the interpreted equations (Definition~\ref{satisf}). 

With this we can prove the characterization theorem below which further expands the results of \cite{Po}, \cite{BG} and~\cite{LP}. In particular we deduce a purely syntactic way to describe enriched categories of algebras of $\lambda$-ary monads on $\cv$. All notions appearing below that are not yet defined shall be introduced in due time.

\begin{theo:char}
	The following are equivalent for a $\cv$-category $\ck$: \begin{enumerate}
		\item $\ck\simeq\Mod(\mathbb E)$ for a $\lambda$-ary equational theory $\mathbb E$ on some $\lambda$-ary language $\mathbb L$;
		\item $\ck\simeq\tx{Alg}(T)$ for a $\cv$-monad $T$ on $\cv$ preserving $\lambda$-filtered colimits;
		\item $\ck$ is cocomplete and has a $\lambda$-presentable and $\cv$-projective strong generator $G\in\ck$;
		\item $\ck\simeq\lambda\tx{-Pw}(\ct^{\op},\cv)$ is equivalent to the $\cv$-category of $\cv$-functors preserving $\lambda$-small powers, for some $\cv_\lambda$-theory $\ct$.
	\end{enumerate}
\end{theo:char}

In Section~\ref{elimination-arities} we discuss which arities are really necessary to express models of equational theories. The main result of the section (Theorem~\ref{elimination}) explains why in the case of $\cv=\Pos,\Met,\omega$-$\CPO$ it is enough to consider terms with trivial output arity, and will be useful for the development of new specific examples, including for instance 2-categorical and simplicial universal algebra.

The second part of the paper is dedicated to proving enriched versions of Birkhoff's variety theorem. As we explain below, to obtain that we shall make some additional assumptions; these involve, for instance, a possibly more general notion of term. 

We shall see that every $(X,Y)$-ary term as defined above corresponds to a morphism $FY\to FX$ between the free $\mathbb L$-structures on the arities $X$ and $Y$. However, one cannot expect that, in general, every morphism of this form can be replaced by one as in \ref{terms}. For this reason we shall call {\em extended $(X,Y)$-ary term} any morphism of the form $FY\to FX$ in the $\cv$-category $\Str(\mathbb L)$ of $\mathbb L$-structures. These are the same as morphisms in the enriched Lawvere theory generated by $\Str(\mathbb L)$ and coincide with the {\em parametrized operations} of~\cite[3.2]{LP} (see Remark~\ref{parametrized-operation}).

Now, by allowing equations between extended terms, we can prove the enriched Birkhoff-type theorem below. The question of whether extended terms below can be replaced by (standard) terms remains open.

\begin{theo:birk}
	Let $\mathbb L$ be a $\lambda$-ary language for which in $\Str(\mathbb L)$ every strong epimorphism is regular. Then the full subcategories of $\Str(\mathbb L)$ closed under products, powers, subobjects, and $\cv$-split quotients are precisely the classes defined by equational $\mathbb L$-theories involving extended terms.
\end{theo:birk}

When $\cv=\Set$ the hypothesis of the theorem are satisfied for any language since the category $\Str(\mathbb L)$ is always regular~\cite[Remark~3.4]{AR}, so that regular and strong epimorphisms coincide. However, for a general $\cv$, whether or not $\Str(\mathbb L)$ satisfies the hypotheses above will depend on which arities are involved in the language $\mathbb L$ itself. We shall see several application of this in Appendix~\ref{more-birk}; in particular our result yields the one from~\cite{Bl}.

In the final Section~\ref{multi-sort} we explore the enriched analogue of multi-sorted universal algebra, introduced by Birkhoff and Lipson~\cite{BL}. As in the single-sorted case, the categorical treatment uses algebraic theories; that is, small categories with finite products whose objects define sorts. Then $S$-sorted equational theories correspond to finitary monads on $\Set^S$ (see for instance \cite[A.40]{AR3}). However, contrary to what one might initially think, the $S$-sorted universal algebra is not (single-sorted) enriched universal algebra over $\Set^S$; since algebras are sets $A$ equipped with $(X,Y)$-ary functions where $X,Y\in\Set^S$, and not $S$-sorted sets $A$. Thus, multi-sorted enriched universal algebra needs a separate treatment, which we establish in Section~\ref{multi-sort} with the main result being Theorem~\ref{multi-char}. Another approach to this problem is given by the recent paper \cite{P}. 

The topics covered in this paper can be further generalized in the directions of \cite{BG,LP0} where one considers a more general class of arities: instead of taking objects of $\cv$ one takes objects of some ambient $\cv$-category $\ck$. However, we preferred to keep the presentation as simple as possible, to provide a gentle introduction to this new topic. We further believe that this work will serve as a starting port for the development and interpretation of new fragments of logic in the context of enriched category theory. Including, for instance, relational languages and regular theories.

\section{Background notions}

\subsection{Enrichment}

As our base of enrichment we fix a symmetric monoidal closed category $\cv=(\cv_0,\otimes,I)$ with internal hom $[-,-]$. When talking about $\mathbb L$-structures (from Section~\ref{languages}) we will denote the internal hom as follows
$$A^X:=[X,A];$$
this is to give a more intuitive interpretation of arities and function symbols.

We assume $\cv$ to be locally $\lambda$-presentable as a closed category~\cite{Kel82}, for some fixed regular cardinal $\lambda$. This means that $\cv_0$ is locally $\lambda$-presentable and the full subcategory $(\cv_0)_\lambda$ spanned by the $\lambda$-presentable objects is closed under the monoidal structure of $\cv_0$. 

Every time we talk about limits and colimits in a $\cv$-category we assume them to be enriched~\cite{Kel82:book}. For the purposes of this paper we shall not use enriched {\em weighted limits} in full generality, but just conical limits and powers (as well as their duals: conical colimits and copowers). 

Conical limits are based on diagrams $H\colon \cd_\cv\to \ck$ out of a free $\cv$-category on a small ordinary one, and into a $\cv$-category $\ck$. The {\em (conical) limit} of such a diagram $H\colon\cd_\cv\to \ck$ is the data of an object $\lim H\in\ck$ together with a cone $\Delta (\lim H)\to H$ inducing an isomorphism
$$ \ck(A,\lim H)\cong [\cd_\cv,\cv](\Delta A,H) $$	
in $\cv$ for any $A\in\ck$; this, when it exists, coincides with the ordinary limit of $H$ in the underlying category $\ck_0$ of $\ck$ (see \cite{Kel82:book}). 

The {\em power} of an object $K\in\ck$ by $X\in\cv$ is the data of an object $K^X\in\ck$ together with a map $X\to\ck(K,K^X)$ inducing an isomorphism
$$ \ck(A,K^X)\cong [X,\ck(A,K)] $$	
in $\cv$ for any $A\in\ck$. 

For any set $S$, the coproduct $S\cdot I$ of copies of the monoidal unit $I$ is called a \textit{discrete} object of $\cv$. For every object $X$ of $\cv$ there is the induced morphism $\delta_X\colon X_0\to X$ where $X_0=\cv_0(I,X)\cdot I$ is discrete.

Given a small full subcategory $\cg$ of a $\cv$-category $\ck$, with inclusion $H\colon \cg\hookrightarrow \ck$, we say that $\cg$ is an (enriched) {\em strong generator} of $\ck$ if the $\cv$-functor $$\ck(K,1)\colon \ck\to[\cg^{\op},\cv]$$ is conservative. Then, following \cite{Kel82}, we say that a $\cv$-category $\ck$ is {\em locally $\lambda$-presentable} if it is cocomplete (all conical colimits and copowers exist) and has a strong generator $\cg$ made of $\lambda$-presentable objects (that is, $\ck(G,-)\colon \ck\to\cv$ preserves $\lambda$-filtered colimits for any $G\in\cg$). 

Finally, (orthogonal) factorization systems will make an appearance in Section~\ref{more-birk}. Following \cite{FK} we will say that a factorization system $(\ce,\cm)$ on a category $\ck$ is {\em proper} if every element of $\ce$ is an epimorphism and every element of $\cm$ a monomorphism. The factorization will be called {\em enriched}, in the sense of \cite{LW}, if the class $\ce$ is closed in $\cv^\to$ under all copowers (if $e\in\ce$ and $X\in\cv$, then $X\otimes e\in\ce$), or equivalently if $\cm$ is closed in $\cv^\to$ under all powers (if $m\in\cm$ and $X\in\cv$, then $[X,m]\in\cm$). 

\subsection{Pretheories, theories, and monads}

Some of our constructions will be related to those considered by Bourke and Garner in~\cite{BG}. In particular their notions of pretheories and theories will be relevant. 

\begin{defi}
	 By a {\em $\cv_\lambda$-pretheory} we mean the data of a $\cv$-category $\ct$ together with an identity-on-object $\cv$-functor $\tau\colon\cv_\lambda^{\op}\to\ct$.
\end{defi}

\begin{rem}
	In the original notation of \cite{BG} a $\cv_\lambda$-pretheory is actually obtained by taking the opposite of the $\cv$-category $\ct$ considered above. We have opted for this change of notation since we care more about the morphisms in $\ct$ rather than in $\ct^{\op}$.
\end{rem}

The $\cv$-category of (concrete) models of a $\cv_\lambda$-pretheory $(\ct,\tau)$ is defined by the pullback
\begin{center}
	\begin{tikzpicture}[baseline=(current  bounding  box.south), scale=2]
		
		\node (a0) at (0,0.8) {$\tx{Mod}(\ct)$};
		\node (a0') at (0.3,0.6) {$\lrcorner$};
		\node (b0) at (1.3,0.8) {$[\ct,\ \cv]$};
		\node (c0) at (0,0) {$\cv$};
		\node (d0) at (1.3,0) {$[\cv_\lambda^{\op},\cv]$};
		
		\path[font=\scriptsize]
		
		(a0) edge [right hook->] node [above] {} (b0)
		(a0) edge [->] node [left] {$U$} (c0)
		(b0) edge [->] node [right] {$[\tau,\cv]$} (d0)
		(c0) edge [right hook->] node [below] {$\cv(K,1)$} (d0);
	\end{tikzpicture}	
\end{center} 
in $\cv\tx{-}\bo{CAT}$; where $K\colon\cv_\lambda\hookrightarrow\cv$ is the inclusion.
Note that this pullback is also a bipullback since $[\tau,\cv]$ is a discrete isofibration. A model of $\ct$ is then an object $A$ of $\cv$ endowed with an extension of $A^{(-)}:=\cv(K-,A)\colon\cv_\lambda^{op}\to\cv$ to a $\cv$-functor $\hat{A}\colon\ct\to\cv$.

By a monad $T$ on $\cv$ we will always mean a $\cv$-monad $T\colon \cv\to\cv$; this is called
\textit{$\lambda$-ary} if $T$ preserves $\lambda$-filtered colimits.
Models of $\cv_\lambda$-pretheories are used in~\cite{BG} to characterize the $\cv$-categories of algebras of $\lambda$-ary monads on $\cv$. It is shown in particular that the forgetful $\cv$-functor $U\colon \tx{Mod}(\ct)\to \cv$ is {\em strictly} $\lambda$-ary monadic; meaning that it has a left adjoint and that the $\cv$-category of algebras of the induced monad is isomorphic to $\tx{Mod}(\ct)$. This is stronger than standard monadicity, which instead requires an equivalence of $\cv$-categories.

On the other hand, every $\lambda$-ary monad $T$ on $\cv$ uniquely identifies a $\cv_\lambda$-pretheory $(\ct,\tau)$ such that $\ct$ has $\lambda$-small powers and $\tau$ preserves them \cite[Section~4.4]{BG}. Those $\cv_\lambda$-pretheories satisfying this additional property are called {\em $\cv_\lambda$-theories}.

It turns out that, to obtain a $\cv_\lambda$-theory, it is enough to ask for the existence and preservation of just the $\lambda$-small powers of the unit $I\in\cv_\lambda$. Such characterization goes back to \cite{Po} (see \cite[Examples 44(iii) and (vi)]{BG}).

\begin{propo}\label{theory}
	An identity-on-objects $\cv$-functor $\tau\colon\cv_\lambda^{\op}\to \ct$ is a $\cv_\lambda$-theory if and only if for any $Z\in\cv_\lambda$
	$$
	\ct(Z,I)^X\cong\ct(Z,X) 
	$$
	$\cv$-naturally in $X\in\cv_\lambda^{\op}$; in other words, if $\ct(Z,-)$ preserves $\lambda$-small powers of $I$. 
\end{propo}
\begin{proof}
	The necessity follows from the fact that if $\ct$ is a $\cv$-theory then $\tau\colon\cv_\lambda^{\op}\to\ct$ preserves powers by all $\lambda$-presentable objects. Conversely, assume that
	$\ct(Z,I)^X\cong\ct(Z,X)$ for every $\lambda$-presentable objects $X$ and $Z$. Recall that $\tau\colon \cv_\lambda^{\op}\to\ct$ is a $\cv$-theory in the sense of \cite{BG} if and only if for each $Z\in\cv_\lambda$ there exists $W\in\cv$ for which
	$$ \ct(Z,\tau-)\cong [K-,W]\colon \cv_\lambda^{\op}\to \cv, $$
in that case $W$ is the free $\ct$-model on $Z$. Now, by hypothesis we have
	\begin{align*}
		\ct(Z,\tau X)\cong \ct(Z,X)\cong \ct(Z,I)^X= [KX,\ct(Z,I)]
	\end{align*}
	where the last equality is simply a change of notation for the internal homs in $\cv$. It follows that $W=\ct(Z,I)$ exists for any $Z$ and hence $\ct$ is a $\cv$-theory. 
\end{proof}

\begin{rem}
	For every $\cv_\lambda$-pretheory $\tau\colon \cv_\lambda^{\op}\to\ct$ and every $Z\in\cv_\lambda$ we have the comparison morphism
	$$
	\gamma_Z\colon \ct(Z,X)\to\ct(Z,I)^X.
	$$
	Hence $\ct$ is a $\cv$-theory if and only if $\gamma_Z$ is an isomorphism for every $Z\in\cv_\lambda$.
\end{rem}

\section{Languages}\label{languages}

In this section we introduce two central notions of this paper; namely those of {\em language} and {\em structure}. For any language $\mathbb L$ we introduce the $\cv$-category $\Str(\mathbb L)$ whose underlying ordinary category has $\mathbb L$-structures as objects and morphisms of $\mathbb L$-structures as arrows. We shall then prove several properties about such $\cv$-categories.

Our notion of language was considered before in \cite{R0,R3,LP}; in \cite[5.1]{LP} it was referred to as a free-form signature.

\begin{defi}
	A single-sorted {\em (functional) language} $\mathbb L$ (over $\cv$) is the data of a set of function symbols $f\colon(X,Y)$ whose arities $X$ and $Y$ are objects of $\cv$.
	The language $\mathbb L$ is called {\em $\lambda$-ary} if all the arities appearing in $\mathbb L$ lie in $\cv_\lambda$.
\end{defi}

Since every language $\mathbb L$ has a small collection of function symbols, it is not restrictive to assume that it is $\lambda$-ary for some big enough $\lambda$. 

We naturally associate a notion of $\mathbb L$-structure to $\mathbb L$. This was also considered in \cite{R0,R3,LP}. 

\begin{defi}
	Given a language $\mathbb L$, an {\em $\mathbb L$-structure} is the data of an object $A\in\cv$ together with a morphism $$f_A\colon A^X\to A^Y$$ in $\cv$ for any function symbol $f\colon(X,Y)$ in $\mathbb L$.

	A {\em morphism of $\mathbb L$-structures} $h\colon A\to B$ is the data of a map $h\colon A\to B$ in $\cv$ making the following square commute
	\begin{center}
		\begin{tikzpicture}[baseline=(current  bounding  box.south), scale=2]
			
			\node (a0) at (0,0.8) {$A^X$};
			\node (b0) at (1,0.8) {$B^X$};
			\node (c0) at (0,0) {$A^Y$};
			\node (d0) at (1,0) {$B^Y$};
			
			\path[font=\scriptsize]
			
			(a0) edge [->] node [above] {$h^X$} (b0)
			(a0) edge [->] node [left] {$f_A$} (c0)
			(b0) edge [->] node [right] {$f_B$} (d0)
			(c0) edge [->] node [below] {$h^Y$} (d0);
		\end{tikzpicture}	
	\end{center} 
	for any $f\colon(X,Y)$ in $\mathbb L$.
\end{defi}

So far $\mathbb{L}$-structures and morphisms between them form just an ordinary category $\Str(\mathbb L)_0$. We shall now produce a $\cv$-category $\Str(\mathbb L)$ whose underlying ordinary category (as the name suggests) will be the one just introduced. This was also done in \cite{LP} following a different approach that provides the same result.

Consider the ordinary category $\cc(\mathbb L)^\lambda$ which has the same objects as $\cv_\lambda$ and whose morphisms are freely generated under composition by the function symbols of $\mathbb L$, so that $f\colon(X,Y)$ in $\mathbb L$ will have domain $X$ and codomain $Y$ in $\cc(\mathbb L)$. Let now $\cc(\mathbb L)_\cv^\lambda$ be the free $\cv$-category on $\cc(\mathbb L)^\lambda$; then we consider the pushout in $\cv\tx{-}\bo{Cat}$

\begin{center}
	\begin{tikzpicture}[baseline=(current  bounding  box.south), scale=2]
		
		\node (a0) at (0,0.8) {$|\cv_\lambda|$};
		\node (b0) at (1.1,0.8) {$\cc(\mathbb L)_\cv^\lambda$};
		\node (c0) at (0,0) {$\cv_\lambda^{op}$};
		\node (d0') at (0.92,0.15) {$\ulcorner$};
		\node (d0) at (1.1,0) {$\Theta_\mathbb L^\lambda$};
		
		\path[font=\scriptsize]
		
		(a0) edge [->] node [above] {$j$} (b0)
		(a0) edge [->] node [left] {$i$} (c0)
		(b0) edge [->] node [right] {$H_\mathbb L$} (d0)
		(c0) edge [->] node [below] {$\theta_\mathbb L^\lambda$} (d0);
	\end{tikzpicture}	
\end{center} 
where $|\cv_\lambda|$ is the free $\cv$-category on the set of objects of $\cv_\lambda$, and $i$ and $j$ are the identity on objects inclusions. It follows that $H_\mathbb L$ and $\theta_\mathbb L^\lambda$ are the identity on objects as well.

The $\cv$-functor $\theta_\mathbb L^\lambda$ defines a $\cv_\lambda$-pretheory whose $\cv$-category of models will be our $\cv$-category of $\mathbb L$-structures:

\begin{defi}
	The $\cv$-category $\Str(\mathbb L)$ on a $\lambda$-ary language $\mathbb L$ is defined as $\Mod(\Theta_\mathbb L^\lambda)$; that is, as the pullback
	\begin{center}
		\begin{tikzpicture}[baseline=(current  bounding  box.south), scale=2]
			
			\node (a0) at (0,0.8) {$\Str(\mathbb L)$};
			\node (a0') at (0.3,0.6) {$\lrcorner$};
			\node (b0) at (1.3,0.8) {$[\Theta_\mathbb L^\lambda,\cv]$};
			\node (c0) at (0,0) {$\cv$};
			\node (d0) at (1.3,0) {$[\cv_\lambda^{op},\cv]$};
			
			\path[font=\scriptsize]
			
			(a0) edge [right hook->] node [above] {} (b0)
			(a0) edge [->] node [left] {$U_\mathbb L$} (c0)
			(b0) edge [->] node [right] {$[\theta_\mathbb L^\lambda,\cv]$} (d0)
			(c0) edge [right hook->] node [below] {$\cv(K-,1)$} (d0);
		\end{tikzpicture}	
	\end{center} 
\end{defi}

An element of $\Str(\mathbb L)$ is then an object $A$ in $\cv$ endowed with an extension of $A^{(-)}:=\cv(K-,A)\colon\cv_\lambda^{op}\to\cv$ to a $\cv$-functor $\hat{A}\colon\Theta^\lambda_\mathbb L\to\cv$. We will see in Proposition~\ref{F-str} below that this is the same data as an $\mathbb L$-structure.

\begin{rem}
	Every $\lambda$-ary language $\mathbb L$ is $\kappa$-ary for any $\kappa\geq\lambda$; however the $\cv$-category $\Str(\mathbb L)$ is independent from the choice of such $\kappa$. Indeed, it is easy to see that for any $\kappa\geq\lambda$, where $\lambda$ is the smallest for which $\mathbb L$ is $\lambda$-ary, we have a pushout square
	\begin{center}
		\begin{tikzpicture}[baseline=(current  bounding  box.south), scale=2]
			
			\node (a0) at (0,0.8) {$\cv_\lambda^{\op}$};
			\node (b0) at (1.1,0.8) {$\Theta_\mathbb L^\lambda$};
			\node (c0) at (0,0) {$\cv_\kappa^{\op}$};
			\node (d0) at (1.1,0) {$\Theta_\mathbb L^\kappa$};
			\node (d0') at (0.92,0.15) {$\ulcorner$};
			\path[font=\scriptsize]
			
			(a0) edge [->] node [above] {$\theta_{\mathbb L}^\lambda$} (b0)
			(a0) edge [right hook->] node [left] {} (c0)
			(b0) edge [->] node [right] {} (d0)
			(c0) edge [->] node [below] {$\theta_{\mathbb L}^\kappa$} (d0);
		\end{tikzpicture}	
	\end{center} 
	in $\cv$-$\Cat$ induced by the definition of $\Theta_\mathbb L^\lambda$ and $\Theta_\mathbb L^\kappa$. (This also applies for $\kappa=\infty$, where $\cv_\infty=\cv$). Therefore, the square on the right in the diagram below is a pullback.
	\begin{center}
		\begin{tikzpicture}[baseline=(current  bounding  box.south), scale=2]
			
			\node (a0) at (0,0.8) {$\Str(\mathbb L)$};
			\node (a0') at (1.5,0.6) {$\lrcorner$};
			\node (b0) at (1.2,0.8) {$[\Theta_\mathbb L^\kappa,\cv]$};
			\node (c0) at (0,0) {$\cv$};
			\node (d0) at (1.2,0) {$[\cv_\kappa^{op},\cv]$};
			\node (e0) at (2.4,0.8) {$[\Theta_\mathbb L^\lambda,\cv]$};
			\node (f0) at (2.4,0) {$[\cv_\lambda^{op},\cv]$};
			
			\path[font=\scriptsize]
			
			(a0) edge [right hook->] node [above] {} (b0)
			(a0) edge [->] node [left] {$U_\mathbb L$} (c0)
			(b0) edge [->] node [left] {$[\theta_\mathbb L^\kappa,\cv]$} (d0)
			(c0) edge [right hook->] node [below] {} (d0)
			(b0) edge [->] node [left] {} (e0)
			(d0) edge [->] node [left] {} (f0)
			(e0) edge [->] node [right] {$[\theta_\mathbb L^\lambda,\cv]$} (f0);
		\end{tikzpicture}	
	\end{center} 
	As a consequence the square on the left is a pullback if and only if the larger square is, proving our claim. Justified by this, we will often omit the superscript $\lambda$ in $\Theta^\lambda_\mathbb L$.	
\end{rem}

\begin{propo}\label{F-str}
	Let $\mathbb L$ be a $\lambda$-ary language; then:
\begin{enumerate}
		\item the underlying category of $\Str(\mathbb L)$ has $\mathbb L$-structures as objects and maps of  $\mathbb L$-structures as morphisms;
		\item $\Str(\mathbb L)$ is a locally $\lambda$-presentable $\cv$-category;
		\item $U_\mathbb L\colon \Str(\mathbb L)\to \cv$ is a strictly monadic right adjoint which preserves $\lambda$-filtered colimits.
	\end{enumerate}
\end{propo}
\begin{proof}
	Conservativity of $U_\mathbb L$ will follow from $(1)$, while $(2)$ and the remainder of $(3)$ are proven in Section~5.3 of \cite{BG}. Thus we are left to prove $(1)$.
	
	By construction, an object of $\Str(\mathbb L)$ is an object $A$ of $\cv$ endowed with a $\cv$-functor $\hat{A}\colon \Theta_\mathbb L^\lambda\to\cv$ whose restriction along $\theta_\mathbb L$ is   $\cv(K-,A)=A^{(-)}$. Now, by definition of $\theta_\mathbb L$, to give the data above is equivalent to give an object $A\in\cv$ together with an ordinary functor $\tilde{A}\colon \cc(\mathbb L)\to \cv_0$ which acts on objects by sending $X$ to $A^X$, for any $X\in\cv_\lambda$. In particular, $\tilde{A}(I)=A$. Since $\cc(\mathbb L)$ is the category generated by the graph on the function symbols of $\mathbb L$; that is exactly the data of an $\mathbb L$-structure.
	
	The same argument applies to morphisms of the underlying category. Just notice that to give a morphism $\gamma\colon (A,\hat{A})\to(B,\hat{B})$ of $\mathbb L$-structures, is the same as giving a map $h\colon A\to B$ in $\cv$ (by fully faithfulness of $\cv(K-,I)$) together with a natural transformation $\eta'\colon \tilde{A}\to \tilde{B}$ such that $\eta'_X=h^X$. (Note that
	$U_\mathbb L(A,\hat{A})=A$ and $U_\mathbb L(\gamma)=h$.)
\end{proof}

\begin{rem}
	Given a language $\mathbb L$, we can define $\mathbb L$-structures in an arbitrary $\cv$-category $\ck$ with powers: an {\em $\mathbb L$-structure} is the data of an object $A\in\ck$ together with  a morphism $f_A\colon A^X\to A^Y$ in $\ck$ for any function symbol $f\colon(X,Y)$ in $\mathbb L$.
	
	A {\em morphism of $\mathbb L$-structures} $h\colon A\to B$ is determined by a map $h\colon A\to B$ in $\cv$ making the usual squares commute in $\ck$ for any $f\colon(X,Y)$ in $\mathbb L$.

	The $\cv$-category of $\mathbb L$-structures in $\ck$ is defined as the pullback
		\begin{center}
			\begin{tikzpicture}[baseline=(current  bounding  box.south), scale=2]
				
				\node (a0) at (0,0.8) {$\Str(\mathbb L)_\ck$};
				\node (a0') at (0.3,0.6) {$\lrcorner$};
				\node (b0) at (1.3,0.8) {$[\Theta_\mathbb L^\lambda,\ck]$};
				\node (c0) at (0,0) {$\ck$};
				\node (d0) at (1.3,0) {$[\cv_\lambda^{op},\ck]$};
				
				\path[font=\scriptsize]
				
				(a0) edge [right hook->] node [above] {} (b0)
				(a0) edge [->] node [left] {$U_\mathbb L$} (c0)
				(b0) edge [->] node [right] {$[\theta_\mathbb L,\cv]$} (d0)
				(c0) edge [right hook->] node [below] {$S$} (d0);
			\end{tikzpicture}	
		\end{center} 
	where $S$ is the transpose of the power functor $\cv_\lambda^{\op}\otimes \ck\to\ck $. Note that $S$ is fully faithful since it has a reflection $T\colon [\cv_\lambda^{op},\ck]\to\ck$ given by evaluating at $I$.
\end{rem}

\section{Terms}\label{terms-section}

We now turn to the notion of {\em $\mathbb L$-term} coming from a language $\mathbb L$. We first introduce an elementary notion of term (Definition~\ref{terms}) built up recursively from the function symbols of $\mathbb L$, the morphisms of $\cv$ (which provide an enriched version of the change of variables), and the closed monoidal structure of $\cv$. These will be essential for characterizing the $\cv$-categories of algebras of $\lambda$-ary monads (Theorem~\ref{char-single}). Then we introduce a more general notion; that of {\em extended term} (Definition~\ref{extended-term}) which we shall use in Section~\ref{birkhoff-section} to prove Birkhoff-type theorems for our languages.

For a language $\mathbb L$, a notion of term was considered in \cite{R0,R3}. We
enrich this concept by allowing power terms (3) below.

\begin{defi}\label{terms}
	
	Let $\mathbb L$ be a $\lambda$-ary language over $\cv$ and $\lambda\leq \kappa\leq\infty$.
	The class of \textit{$\kappa$-ary $\mathbb L$-terms} is defined recursively as follows:
	\begin{enumerate}
		\item Every morphism $f\colon Y\to X$ of $\cv_\kappa$ is an $(X,Y)$-ary term; 
		\item Every function symbol $f:(X,Y)$ of $\mathbb L$ is an $(X,Y)$-ary term;
		\item If $t$ is a $(X,Y)$-ary term and $Z$ is in $\cv_\kappa$, then $t^Z$ is a $(Z\otimes X,Z\otimes Y)$-ary term;
		\item Given $t_J=(t_j)_{j\in J}$, where $|J|<\kappa$ and $t_j$ is an $(X_j,Y_j)$-ary term, and $s$ an $(\sum_{j\in J} Y_j,W)$-ary term; then $s(t_J)$ is a $(\sum_{j\in J} X_j, W)$-ary term.
	\end{enumerate}
	Let $A$ be an $\mathbb L$-structure, then the {\em interpretation} of $\mathbb L$-terms is defined recursively as follows:
	\begin{enumerate}
		\item Every morphism $f\colon Y\to X$ of $\cv_\lambda$ is interpreted as $$f_A:= A^f\colon A^X\to A^Y;$$
		\item Every function symbol $f:(X,Y)$ of $\mathbb L$ is interpreted as the map $$f_A\colon A^X\to A^Y$$ given by the fact that $A$ is an $\mathbb{L}$-structure;
		\item If $t$ is a $(X,Y)$-ary term and $Z$ is an arity, then $t^Z$ is interpreted as the map 
		$$ t^Z_A\colon A^{Z\otimes X}\to A^{Z\otimes Y}$$
		given by composing $(t_A)^Z\colon (A^X)^Z\to (A^Y)^Z$ with the canonical isomorphisms $(A^X)^Z\cong A^{Z\otimes X}$ and $(A^Y)^Z\cong A^{Z\otimes Y}$;
		\item If $t_J=(t_j)_{j\in J}$ is formed by $(X_j,Y_j)$-ary terms, and $s$ is a $(\sum_{j\in J}Y_j,W)$-ary term, then 
		$s(t_J)$ is interpreted as the composite
		$$ A^{\sum_{j\in J} X_j}\cong \textstyle\prod_i A^{X_j}\xrightarrow{\ \prod_j(t_j)_A\ } \textstyle\prod_j A^{Y_j}\cong A^{\sum_{j\in J}Y_j}\xrightarrow{\ s_A\ } A^W.$$
	\end{enumerate}
\end{defi}	
	
\begin{rem}\label{simpler-language}$ $
	
(1) If we take $s=\id_W$ in (4), we get the term $t_J$.

(2)	Note that if $\cg$ generates $\cv_\kappa$ under $\kappa$-small coproducts, then we can assume the output arities of our enriched languages to lie in $\cg$. This is because an general $(X,Y)$-ary symbol $f$ in a language $\mathbb L$ can be replaced by a family of function symbols $f_j$ of arity $(X,Y_j)$ with $Y_j\in\cg$, $j\in J$ and $Y=\sum_j Y_j$.

Call $\mathbb L'$ the language obtained by $\mathbb L$ by applying this operations. Then $\mathbb L$-structures and $\mathbb L'$-structures are the same (universal property of (co)products) and, thanks to rule (4) above, $\mathbb L$-terms and $\mathbb L'$-terms are equivalent.
In more detail, $f$ is given by the composition of $f_J$ with the term given by the codiagonal $\nabla:\sum_j X\to X$.

(3) If $t$ is an $(X,Y)$-ary term and $Z=\sum_{j\in J}I$ where
$|J|<\kappa$ then the terms $t^Z$ and $(t_j)_{j\in J}$ where $t_j=t$
for every $j\in J$ have the same interpretation on every $\mathbb L$-structure.
\end{rem}	
	
In the first point of the following example we explain the correspondence, in the ordinary case, between our notions of structures and terms and those of universal algebra. In the remaining points we make connections with other works in the literature. See also Example~\ref{ex}.

\begin{exams}\label{extended}$ $
	{\setlength{\leftmargini}{1.6em}
		\begin{enumerate}
			\item  In universal algebra, a \textit{signature} $\Sigma$ is a set $\mathbb L$ of finitary function symbols. Any such signature $\Sigma$ is a language in our sense (over $\cv=\Set$) where $n$-ary function symbols are $(n,1)$-ary ones. Conversely, given a finitary language $\mathbb L$ in our sense; this corresponds to a signature in the ordinary sense by Remark~\ref{simpler-language} above. 
			
			Concerning terms; ordinarily these are formed from variables and function symbols by applying superpositions $f(t_1,\dots,t_m)$. In our setting, variables are dealt with in rule (1): a map $g\colon m\to n$ between finite sets, corresponds to the $m$-tuple of $n$-ary terms given by the projections $\pi_{g(i)}(x_1,\dots,x_n)$ to the $g(i)$-th variable. In particular, the identity on $n$ declares variables $(x_1,\dots,x_n)$. Then, rule (2) adds function symbols and (4) generates under superpositions. Terms from rule (3) are superfluous in this case since they correspond to tuples of the form $(t,\dots,t)$ which are already introduced in (4) --- it will follow from Proposition~\ref{G-powers} that this rule can be avoided from the beginning. 
			
			Thus, $(n,1)$-ary terms correspond to $n$-ary terms in the sense of universal algebra; while, in general, an $(n,m)$-ary term is a $m$-tuple of $n$-ary terms. The fact that only $(n,1)$-terms are necessary to do ordinary universal algebra will follow from Proposition~\ref{G-output}. 
			
			\item Let $\cv=\Pos$ be the cartesian closed category of posets and monotone maps. It is locally finitely presentable as a closed category and finitely presentable objects are finite posets. Since the terminal object is a generator, by Proposition~\ref{G-output}, also in this context it is enough to use $(X,1)$-ary terms, where $X$ is a finite poset.
			
			A {\em signature in context} from \cite[Definition~3.2]{AFMS} is the same as a finitary language $\mathbb L$ in our sense with function symbols of arity $(X,1)$. A {\em coherent $\mathbb L$-algebra} of \cite{AFMS} is just an $\mathbb L$-structure in our sense. 
			
			Terms in \cite{AFMS} are just terms from ordinary universal algebra: an $X$-ary term, for $X$ a finite poset, is defined as an $n$-ary ordinary term, where $n$ is the cardinality of $X$ (\cite[3.10]{AFMS}). Thus any $(X,1)$-term $t$ in our setting corresponds to a term $t_0$ of arity $X$ in their setting. While, given a term $s$ from \cite{AFMS}, if the rules applied to define $s$ preserve the order of the arities involved, then $s=t_0$ for some $(X,1)$-ary term $t$. 
			
			Note that if a signature $\mathbb L$ is not in context (that is, arities are discrete posets), we still have $(X,1)$-ary terms for every finite poset $X$ -- they appear as superpositions of $(X,X_0)$-ary terms given by $\delta_X\colon X_0\to X$
			with $(X_0,1)$-ary terms. Here $\delta_X$ is the identity from the discrete poset $X_0$ on the set $X$ to the poset $X$.
			
			
			\item Let $\Met$ be the symmetric monoidal closed category of generalized metric spaces 
			(distance $\infty$ is allowed) and nonexpanding maps 
			(\cite[2.2(1)]{AR1}). This is a locally $\aleph_1$-presentable category
			and $\aleph_1$-presentable objects are countable metric spaces. Also in this case, by Proposition~\ref{G-output}, it is enough to use $(X,1)$-ary terms.
			
			A {\em signature} $\mathbb L$ {\em in context} in \cite{MPP,MPP1,A} is the same as an $\aleph_1$-ary language $\mathbb L$ with function symbols of arity $(X,1)$. The situation is similar to (2): $\mathbb L$-algebras are $\mathbb L$-structures in our sense; the correspondence between our terms and those from \cite{MPP,MPP1,A} is similar to that for $\cv=\Pos$ (replace ordering with distance, and finite with countable). 
			
			
			\item  Let $\omega$-$\CPO$ be the cartesian closed category of posets with joins of non-empty $\omega$-chains and maps preserving joins of non-empty $\omega$-chains
			(see \cite[2.9]{ADV}); these maps are called continuous. This category is locally $\aleph_1$-presentable and $\aleph_1$-presentable objects are countable cpo's. 
			
			In \cite{ANR,ADV}, a signature is the same as our language $\mathbb L$ with function symbols of arity $(X,1)$ where $X$ is countable discrete (that is, a countable antichain). Our $\mathbb L$-structures coincide with their continuous algebras. Terms from our setting and from \cite{ANR,ADV} can be compared again as in the previous points.
			
			\item Let $\Ab$ be the symmetric monoidal closed category of abelian groups. Since $\mathbb Z$ is a generator, by Proposition~\ref{G-output}, we can again reduce to $(X,1)$-ary terms, where $X$ is finitely presented. 
			Finitely presented abelian groups are finite direct sums of copies of $\mathbb Z=\mathbb Z/0\mathbb Z$ and $\mathbb Z/n\mathbb Z$, for some $n>0$. Thus, it follows from Remark~\ref{simpler-language} that a finitary language $\mathbb L$ over $\Ab$ has function symbols of arity $(\oplus_{i=1}^k\mathbb Z/n_i\mathbb Z,\mathbb Z/m\mathbb Z)$, for $n_i,m\geq 0$, interpreted as morphisms
			$$ \bigoplus_{i=0}^n A_{n_i}\longrightarrow A_m, $$
			where $A_n=A^{\mathbb Z/n\mathbb Z}$ is the subgroup of $A$ spanned by those $a$ for which $na=0$.  
			
			Terms are generated by function symbols, variables, and ordinary superposition, plus the following two operations:
			
			\begin{enumerate}
				\item If $f$ and $g$ are $(X,Y)$-ary terms, then there is an $(X,Y)$-term $f+g$ which is interpreted as the sum of the interpretations of $f$ and $g$ (this is obtained using the substitution rule applied to $(f,g)$, the codiagonal of $X\oplus X$, and the diagonal of $Y\oplus Y$).
				\item If $f$ is a $(X,Y)$-ary term, then there is an $(X,Y)$-ary term $-f$ which is interpreted as the opposite of $f$ (apply the substitution rule to $f$ and the $(X,X)$-ary term $-\tx{id}$).
			\end{enumerate}
		It is easy to see that structures and terms over the empty language (over $\Ab$, but also over $\cv=R$-$\Mod$) can be interpreted within the framework of \cite{Pre03,Pre09}, but not in the equational fragment of that theory. This is because to express that a function symbol (or term) is defined out of $A_n$, rather than $A$, one needs to use implications between equations.
		\end{enumerate}
	}
\end{exams}

The terms just introduced will be enough for our characterization theorems of Section~\ref{equations} and will allow us to create connections with previous work in the literature, as well as to express various examples. However, we do not know whether they suffice in general for the Birkhoff variety theorems of Section~\ref{birkhoff-section}. That is why we shall also introduce extended terms.

Following \cite{BG}, a $\lambda$-ary language $\mathbb L$ yields the $\cv_\lambda$-theory $\tau_\mathbb L^\lambda\colon \cv_\lambda^{\op} \to\ct_\mathbb L^\lambda$ induced by $\Theta_{\mathbb L}^\lambda$. The $\cv$-category $(\ct_\mathbb L^\lambda)^{\op}$ can be obtained by taking the (identity-on-objects, fully faithful) factorization below 
$$ \cv_\lambda\xrightarrow{i.o.o.} (\ct_\mathbb L^\lambda)^{\op} \xrightarrow{f.f.} \Str(\mathbb L) $$
of the free $\cv$-functor $F\colon \cv\to \Str(\mathbb L)$ preceded by the inclusion $K\colon \cv_\lambda\hookrightarrow\cv$. 
	
Note also that there is an identity-on-objects functor $\Gamma _\mathbb L^\lambda\colon \Theta_\mathbb L^\lambda\to \ct_\mathbb L^\lambda$ such that 
\begin{center}
	\begin{tikzpicture}[baseline=(current  bounding  box.south), scale=2]
		
		\node (a0) at (0,-0.8) {$ \Theta_\mathbb L^\lambda$};
		\node (b0) at (1.3,-0.8) {$\ct_\mathbb L^\lambda$};
		\node (c0) at (0.65,0) {$\cv_\lambda^{\op}$};
		
		\path[font=\scriptsize]
		
		(a0) edge [->] node [below] {$\Gamma _\mathbb L^\lambda$} (b0)
		(a0) edge [<-] node [left] {$\theta_\mathbb L^\lambda$} (c0)
		(b0) edge [<-] node [right] {$\tau_\mathbb L^\lambda$} (c0);
	\end{tikzpicture}	
\end{center} 
commutes.
	
We allow this construction also for $\lambda=\infty$; in this case $(\ct_\mathbb L^\infty)^{\op}$ is obtained by taking the (identity-on-objects, fully faithful) factorization below 
$$ \cv\xrightarrow{i.o.o.} (\ct_\mathbb L^\infty)^{\op}\xrightarrow{f.f.} \Str(\mathbb L). $$
of $F\colon\cv\to\Str(\mathbb L)$.


\begin{defi}\label{extended-term}
	Let $\mathbb L$ be a $\lambda$-ary language, and $\lambda\leq\kappa\leq\infty$. An {\em extended $\kappa$-ary term $t:(X,Y)$ for $\mathbb L$} is a morphism $t\colon X\to Y$ in $\ct_\mathbb L^\kappa$. Equivalently, an extended $\kappa$-ary term $t:(X,Y)$ is just a morphism $t\colon FY\to FX$ in $\Str(\mathbb L)$.
\end{defi}

The interpretation of such $t:(X,Y)$ on an $\mathbb L$-structure $A$ is given by the composite
$$ t_A\colon A^X\xrightarrow{\cong}\Str(\mathbb L)(FX,A)\xrightarrow{\Str(\mathbb L)(t,A)} \Str(\mathbb L)(FY,A)\xrightarrow{\cong}A^Y.$$

\begin{rem}
	If $s:(X,Y)$ and $t:(Y,Z)$ are terms given as in point (1) and/or (2) of Definition~\ref{terms}, then they can be identified with morphisms of $\Theta_{\mathbb L}^\lambda$ (and hence, by the arguments above, of) $\ct_\mathbb L^\kappa$, and hence as extended terms. This is well-defined since their composition as morphisms of $\ct_\mathbb L^\kappa$ has the same interpretation as their composition as terms. Then, arguing recursively, a $\kappa$-ary term $t$ from rule (3) can be seen as a morphism in $\ct_\mathbb L^\kappa$; under this correspondence, the term $t^Z$ corresponds to the power of $t$ by $Z$ in $\ct_\mathbb L^\kappa$. Similarly, in rule (4), the term $s(t_I)$ and the extended term $s\circ (\prod_it_i)$ have the same interpretation. 
	It follows that every term can be naturally seen as an extended term.
\end{rem}

\begin{rem}\label{parametrized-operation}
	In \cite[Definition~3.2]{LP} a {\em parametrized operation}, restricted to our specific setting, is defined as a $\cv$-natural transformation
	$$ U(-)^X\longrightarrow U(-)^Y $$
	where as usual $U\colon \Str(\mathbb L)\to \cv$ is the forgetful $\cv$-functor and $X,Y$ are objects of $\cv$. By adjointness, this corresponds to a $\cv$-natural $\Str(\mathbb L)(FX,-)\to \Str(\mathbb L)(FY,-)$, which in turn is just a map $FY\to FX$ in $\Str(\mathbb L)$. Thus our extended terms and their parametrized operations coincide.
\end{rem}

\section{Enriched equational theories}\label{equations}

We can now introduce {\em equational theories} as collections of equalities $(s=t)$ between terms (or extended terms) of the same arity. Their models will characterize the $\cv$-categories of algebras of $\lambda$-ary monads.
 	
\begin{defi}\label{satisf}
	An \textit{equation} between extended terms is an expression of the form $$(s=t),$$ where $s$ and $t$ are extended terms of the same arity. We say that an $\mathbb L$-structure $A$ {\em satisfies} such equation if $s_A=t_A$ in $\cv$. 
	
	Given a set $\mathbb E$ of equations in $\mathbb L$, we denote by $\Mod(\mathbb E)$ the full subcategory of $\Str(\mathbb L)$ spanned by those $\mathbb L$-structures that satisfy all equations in $\mathbb E$; we call these $\mathbb L$-structures {\em models of $\mathbb E$} and call $\mathbb E$ an {\em extended $\infty$-ary equational theory}. If all the extended terms appearing in $\mathbb E$ are $\lambda$-ary, we call $\mathbb E$ a {\em extended $\lambda$-ary equational theory}.
	
	When $\mathbb E$ consists just of standard (recursively defined) terms, we drop the word {\em extended} and call $\mathbb E$ simply an {\em $\infty$-ary equational theory}, or {\em $\lambda$-ary equational} if the terms are all $\lambda$-ary.
\end{defi}

Using that we can see terms $s$ and $t$ as maps $s,t\colon FY\to FX$ in $\Str(\mathbb L)$, an $\mathbb L$-structure $A$ satisfies the equation $(s=t)$ if and only if $ \Str(\mathbb L)(s,A)=\Str(\mathbb L)(t,A) $ in $\cv$.

\begin{rem}\label{unenrich-sat}
	Note that our satisfaction is in a strong enriched sense: if $A$ satisfies the equation
	$(s=t)$ then $A$ satisfies the equations $(s^Z=t^Z)$ for all $Z$. The unenriched satisfaction of $(s=t)$ would instead mean that $\cv_0(I,s_A)=\cv_0(I,t_A)$, or equivalently $ \Str(\mathbb L)_0(s,A)=\Str(\mathbb L)_0(t,A) $ (seeing $s$ and $t$ as maps $s,t\colon FY\to FX$).
	Then, it is easy to see that enriched satisfaction of $(s=t)$ is equivalent to the unenriched satisfaction of $(s^Z=t^Z)$ for all $Z$. If $I$ is a generator in $\cv_0$ the enriched and unenriched satisfactions are the same.
\end{rem}

Let us now give some examples of equational theories built using terms as in Definition~\ref{terms}.

\begin{exam}
	Let $\cv$ be the cartesian closed category $\MGra$ of directed multigraphs (this is the presheaf category over the two parallel arrows). For any integer $n\geq 0$ let $[n]$ be the graph 
	$$ \{0\to 1\to \cdots \to n-1\}. $$
	So $[1]$ is the terminal object, and $[2]$ is the free edge. We now construct a language and an equational theory for small categories. Consider the language $\mathbb L$ given by function symbols:\begin{itemize}
		\item $I:([1],[2])$ for identities;
		\item $J_i:([2],[3])$ for ``pairing with identities'', $i=1,2$;
		\item $M:([3],[2])$ for the composition map;
		\item $M_1,M_2:([4],[3])$ for ``composing to the left/right''.
	\end{itemize}
	Then define a theory $\mathbb E$ with axioms:\begin{enumerate}
		\item $ (\pi_2 (J_1)= I( \pi_{cod}))$ and $ (\pi_1 (J_1)= \id)$, where $\pi_i :([3],[2])$ and $\pi_{cod}:([2],[1])$ are the terms corresponding respectively to the inclusion $[2]\to[3]$ of the i-th edge and to the codomain inclusion $[1]\to [2]$. 
		\item $ (\pi_1 (J_2)= I (\pi_{dom}))$  and $ (\pi_2(J_2)= \id)$, dual to the above.
		\item $ (M( J_1)=\id)$ and $(M( J_2)=\id)$.
		\item $(q( M_1)= (M, \id)( q_1) )$ and $(q( M_2)= (\id, M)( q_2))$, where $q,q_1,$ and $q_2$ are the terms corresponding to the maps $q\colon [2]+[2]\to [3]$, $q_1\colon [3]+[2]\to [4]$, and $q_1\colon [2]+[3]\to [4]$ obtained by gluing the codomain of the first component with the domain of the second.
		\item $(M (M_2)=M( M_1))$.
	\end{enumerate} 
	Given a model $C$ of $\mathbb E$, the map $I$ assigns an identity edge $1_c$ to any vertex $c$, and $M$ gives a composition rule for any composable pair of edges in $C$. Then the equation (1) says that $J_1\colon C^{[2]}\to C^{[3]}$ sends any edge $f\colon s\to t$ to the pair $(f,1_t)$; while (2)  says that $J_2$ sends $f$ to $(1_s,f)$. Then (3) says that the identities are neutral elements for the composition rule (on both sides).  The axioms in $(4)$ say that $M_1(f,g,s)= (M(f,g),s)$ and $M_2(f,g,s)=(f,M(g,s))$. Finally, (5) says that the composition rule is associative.
	
	It follows that $C$ is a model of $\mathbb E$ if and only if it is equipped with the structure of a category.
\end{exam}

\begin{exam}
	Consider the language $\mathbb L$ over $\Met$ with one $(2,2_1)$-ary function symbol $f$, where $2$ is a two-point metric space whose points have the distance $\infty$, and $2_1$ is a two-point metric space whose points have the distance $1$. Let the theory $\mathbb E$ be given by the equation
	$$ (f (\delta_{2_1})=\id_2) $$
	where $\delta_{2_1}$ is the $(2_1,2)$-ary term given by the bijection $2\to 2_1$ in $\Met$.
	
	Then, a model of $\mathbb E$ is a metric space $A$ together with a map $\delta_A\colon A\times A\to A^{2_1}$ such that $\delta_A(x,y)=(x,y)$. Such a map is well defined if and only if $d(x,y)\leq 1$ for any $(x,y)$ in $A$. Thus models of $\mathbb E$ are metric spaces with distance at most $1$. 
\end{exam}

\begin{exam}
	Let $\ch$ be a collection of morphisms in $\cv_\lambda$; note that any $h\colon X\to Y$ in $\ch$ defines a term $h:(Y,X)$ in any $\lambda$-ary language. Consider then the language $\mathbb L$ consisting of a function symbol $h^{-1}:(X,Y)$ for any $h\colon X\to Y$ in $\ch$, and define the theory $\mathbb E$ with axioms
	$$ (h (h^{-1})= \id)\ \ \text{ and } \ \ (h^{-1}( h)=\id) $$
	for any $h\in\ch$. Then a model of $\mathbb E$ is an object $A\in\cv$ together with maps $h^{-1}_A\colon A^X\to A^Y$ that are inverses of $A^h\colon A^Y\to A^X$. It follows that $\Mod(\mathbb E)\simeq \ch^\perp$ is the $\cv$-category of objects orthogonal with respect to $\ch$ in the enriched sense. (The previous example falls into this setting.)
	
If we define the theory $\mathbb E'$ with axioms
	$$ (h^{-1}( h)=\id) $$
	for any $h\in\ch$, then $\Mod(\mathbb E')$ is the $\cv$-category of algebraic $\ch$-injective objects (see \cite[2.3]{Bo}).
\end{exam}

\begin{exam}\label{kant}
	Consider the following language $\mathbb L$ over $\Met$ defined by:\begin{itemize}  
		\item a function symbol $c_\lambda\colon (1+1,1)$ for each $\lambda\in[0,1]$;
		\item a function symbol $r_\epsilon^\lambda\colon (2_\epsilon+1, 2_{\lambda\epsilon})$, for any $\epsilon>0$ and $\lambda\in[0,1]$.
	\end{itemize}
	Here, $2_\eps$ is a two-point metric space whose points have distance $\eps$ for $\eps>0$ and $2_0=1$ is the one-point metric space. 
	Then define a theory $\mathbb E$ with axioms the equalities in (a)--(d) of \cite[5.2]{FP} (giving convexity conditions) plus an axiom specifying that
	$$ r_\epsilon^\lambda (x,y,z) = (c_\lambda(x,y), c_\lambda(x,z)) $$
	for any $\epsilon>0$.
	
A model of $\mathbb E$ is given by a metric space $A$ together with operations $c_\lambda\colon A\times A\to A$ and maps $r_\epsilon^\lambda\colon A^{2_\epsilon}\times A\to A^{2_{\lambda\epsilon}}$ for any $\epsilon>0$ and $\lambda\in[0,1]$. The $c_\lambda$ are subject to the axioms of \cite[5.2]{FP} which make $A$ into a convex space. The last axioms plus the fact that $r_\epsilon^\lambda$ is a contraction, say that for any triple $(x,y,z)$ in $A$ with $d(x,y)\leq \epsilon$, then $d(c_\lambda(x,z),c_\lambda(y,z))\leq \lambda\epsilon$. It follows that the existence of $r_\epsilon^\lambda$ for any $\epsilon$ and $\lambda$ as above, is equivalent to the following inequality
	$$ d(c_\lambda(x,z),c_\lambda(y,z))\leq \lambda d(x,y) $$
	being true for any $x,y,z$ in $A$ and $\lambda\in[0,1]$. Convex spaces satisfying this condition are studied in \cite[5.4]{FP} where it is shown that the $\Met$-category of $\mathbb E$-models corresponds to that of algebras for the Kantorovich monad.
\end{exam}

Next we provide some examples where power terms are useful. For simplicity, we will often consider the composition $(t\circ s)$ of terms $s:(X,Y)$ and $t:(Y',Z)$ where $Y$ and $Y'$ are isomorphic and the isomorphism $i\colon Y\to Y'$ is clear from the context; then $(t\circ s)$ should be interpreted as $(t\circ i\circ s)$. For instance we do this whenever $Y'=Y\otimes I$, where $I$ is the unit, and the isomorphism is given by the monoidal structure on $\cv$. 

\begin{exam}
	Let $\cv=\bo{GAb}$ be the monoidal closed category of graded abelian groups. Let $P_i$ be the object with $\mathbb Z$ in degree $i$ and $(0)$ otherwise, so that for $i=1$ and any $A\in \bo{GAb}$ we have $(A^{P_1})_n=A_{n+1}$.
	
Consider the language with one function symbol $d:(P_1,I)$. Here we can construct the power term $d^{P_1}$ that has arity $(P_1\otimes P_1,P_1)$; it follows that the output arity of $d^{P_1}$ is the same as the input arity of $d$, so that we can form the new term $d (d^{P_1})$ as per rule $(4)$ of~\ref{terms}. Thus we can define the equational theory $\mathbb E$ given by the single equation 
	$$ d (d^{P_1})=0. $$
	Then $\Mod(\mathbb E)=\bo{DGAb}$ is the category of differentially graded abelian groups. Indeed, to give $d_A\colon A^{P_1}\to A$ in $\bo{GAb}$ is the same as giving a differential $d_A^{n+1}\colon A_{n+1}\to A_n$ for any $n$. Then $d_A$ satisfies the equation of $\mathbb E$ if and only if the composites of the differential are $0$.
\end{exam}

\begin{exam}
	Let $\cv=\bo{GAb}^+$ be the category of positively graded abelian groups, and $P_i$ as above (for $i\geq 0$); note that we have canonical isomorphisms $\sigma_{ij}\colon P_i\otimes P_j\to P_{i+j}$. Consider a graded ring $R=\bigoplus_{i\geq 0}R_i$ (meaning that $R$ is a ring and the multiplication satisfies $R_i\cdot R_j\subseteq R_{i+j}$). Consider the language $\mathbb L$ given by symbols $\hat r:(I,P_i)$ for any $i\geq 0$ and $r\in R_i$. Then define the theory $\mathbb E$ with equations
	$$ (\sigma_{ij} (\hat r^{P_j})( \hat s) =  \widehat{r\cdot s})$$
	for any $\hat r:(I,P_i)$ and $\hat s:(I,P_j)$. Then, the models of $\mathbb E$ are {\em graded $R$-modules}: that is, graded abelian groups $M$ together with associative scalar multiplications $R_i\oplus M_j\to M_{i+j}$. 
	Doing the same for $\cv=\bo{DGAb}$ we obtain the differentially graded $R$-modules.
\end{exam}

\begin{exam}
	Let $\cv=\Gra$ be the cartesian closed category of graphs; that is, sets $V$ (of vertices) equipped with a symmetric binary relation $E$. If $(x,y)\in E$ we say that $(x,y)$ is an edge. Morphisms $(V,E)\to (V',E')$ are mappings $V\to V'$ preserving edges. Recall that $(V,E)^{(V',E')}$ 
	has as vertices all maps $f\colon V'\to V$ and $(f,g)$ is an edge if and only if
	$$
	(x,y)\in E' \Rightarrow (fx,gy)\in E.
	$$
	Let $1$ be a graph with a single vertex and no edge. Then $(V,E)^1$ is the \textit{complete graph} $(V,V\times V)$ and $1\times (V,E)$ is the \textit{edgeless} graph $(V,\emptyset)$. The tensor unit $I$ is the graph with a single vertex and a single edge. Consider the language consisting of an $(I,I)$-ary function symbol $f$. Then the equational theory with
	$$
	(f^1=\id)
	$$
	gives as models graphs with a unary operation which is the identity on vertices.
\end{exam}

\begin{exam}
	Let $\cv=\Cat$ be the category of small categories with its cartesian closed structure, and let ${\bf 2}=\{ 0\to 1\}$ be the arrow category. Consider the language $\mathbb L$ with only one $(1,{\bf 2})$-ary function symbol $\sigma$, so that an $\mathbb L$-structure is the data of a small category $\cc$ together with $$\sigma_\cc\colon \cc\to\cc^{\bf 2}.$$ Consider now the two inclusions $i_0,i_1\colon 1\to \bf{2}$, and define (for simplicity of notation) the terms $S:= i_0\circ \sigma $ and $T:= i_1\circ \sigma $. Then an $\mathbb L$-structure is a small category $\cc$ together with a natural transformation $\hat\sigma\colon S\Rightarrow T\colon\cc\to\cc$.
	
	Let $\mathbb E$ be the theory with equations
	$$ (S=\tx{id})\ \ \text{ and } \ \  (T^{\bf 2}(\sigma)=\sigma (T)); $$
	then a model of $\mathbb E$ is a well-pointed endofunctor; that is, a functor $T\colon \cc\to\cc$ together with a natural transformation $\hat\sigma\colon 1_\cc\Rightarrow T$ such that $T\hat\sigma=\hat\sigma T$.
\end{exam}

\begin{exam}
	Let $\cv=\bo{SSet}$ the cartesian closed category of simplicial sets (which is locally finitely presentable). We denote by $J:=\Delta[1]\in\bo{SSet}$ the free 1-simplex; this comes together with the two boundary maps $j_0,j_1\colon 1\to J$. 
	
	Consider the language $\mathbb L$ with the following function symbols:\begin{itemize}
		\item $x_0,x_1:(0,1)$;
		\item $p:(0,J)$.
	\end{itemize}
	On this language we define the equational theory $\mathbb E$ with axioms
	$$( p ( j_i) = x_i )$$
	for $i=0,1$. Then a model of $\mathbb E$ is the data of a simplicial set $A$ together with two points (vertices) $x_0,y_0\in A$ and a path (edge) $p\colon J\to A$ between them.
	
	One can argue similarly by taking the $\cv$ to be the category $\bo{Gpd}$ of groupoids and $J=\{\cdot \cong\cdot\}$ as the interval object; in this case a path between two objects in a groupoid is simply an isomorphism between them. 
	
	Both $\bo{SSet}$ and $\bo{Gpd}$ are examples of categories where one can interpret {\em intensional type theory} by using such interval objects (see~\cite{AW09}); under this interpretation a model $(A,x_0,x_1,p$) of $\mathbb E$ above provides a type-theoretic proof $p$ of the fact that $x_0$ and $x_1$ are (intensionally) equal. This raises the question of whether it is possible to interpret more complex type-theoretic formulas (in $\bo{SSet},\bo{Gpd}$, or other categories that model type theory) within the framework of our enriched equations, providing a meaningful connection between the two theories.
\end{exam}

\subsection{Main results}

We now turn to study the main properties of the $\cv$-categories of models of equational theories. The result below can also be seen as a consequence of \cite[5.20]{LP}

\begin{propo}\label{eqations->monad}
	For any extended $\lambda$-ary equational theory $\mathbb E$ the $\cv$-category $\Mod(\mathbb E)$ is locally $\lambda$-presentable and the forgetful $\cv$-functor $U\colon \Mod(\mathbb E)\to \cv$ is $\lambda$-ary and strictly monadic.
\end{propo}
\begin{proof}
	We shall use that the 2-category of locally $\lambda$-presentable $\cv$-categories, continuous and $\lambda$-ary $\cv$-functors, and $\cv$-natural transformations, has all flexible limits \cite[Theorem~6.10]{Bird} and hence all bilimits (see for instance~\cite[Section~6]{La}); in particular it has wide bipullbacks.
	
	Consider an equation $(s=t)$ between extended $(X,Y)$-ary terms $s$ and $t$. Recall that $s$ and $t$ can be viewed as morphisms $\hat s,\hat t\colon FY\to FX$ in $\Str(\mathbb L)$, and evaluation at an $\mathbb L$-structure $A$ can be obtained by homming into $A$ (see after Definition~\ref{extended}); it follows that we have $\lambda$-ary $\cv$-functors
	$$ s_{(-)},t_{(-)}\colon \Str(\mathbb L)\to \cv^\to $$
	given, up to isomorphism, by $\Str(\mathbb L)(\hat s,-)$ and $\Str(\mathbb L)(\hat t,-)$. These send an $\mathbb L$-structure $A$ to $s_A,t_A\colon A^X\to A^Y$, as objects of $\cv^\to$. Since $s$ and $t$ have the same input and output arities, the two $\cv$-functors above form a co-fork with the projection $\cv^\to\to\cv\times\cv$; thus, they assemble into a $\lambda$-ary $\cv$-functor
	$$ (s,t)_{(-)}\colon \Str(\mathbb L) \to \cv^{\rightrightarrows}, $$	
	where $\rightrightarrows$ denotes the free $\cv$-category on a parallel pair of arrows. It follows that the $\cv$-category $\Mod(s=t)$ can be seen as the pullback below.
	\begin{center}
		\begin{tikzpicture}[baseline=(current  bounding  box.south), scale=2]
			
			\node (a0) at (0,0.8) {$\Mod(s=t)$};
			\node (a0') at (0.3,0.6) {$\lrcorner$};
			\node (b0) at (1.3,0.8) {$\cv^\to$};
			\node (c0) at (0,0) {$\Str(\mathbb L)$};
			\node (d0) at (1.3,0) {$\cv^{\rightrightarrows}$};
			
			\path[font=\scriptsize]
			
			(a0) edge [->] node [above] {} (b0)
			(a0) edge [->] node [left] {} (c0)
			(b0) edge [->] node [right] {$\Delta$} (d0)
			(c0) edge [->] node [below] {$(s,t)_{(-)}$} (d0);
		\end{tikzpicture}	
	\end{center} 
	Note that this is also a bipullback since $\Delta$ is a discrete isofibration. Therefore $\Mod(s=t)$ is locally $\lambda$-presentable and the inclusion into $\Str(\mathbb L)$ is continuous and $\lambda$-ary.
	
	Now, for a general equational theory $\mathbb E$, it follows that the $\cv$-category $\Mod(\mathbb E)$ is locally $\lambda$-presentable as a full subcategory of $\Str(\mathbb L)$, being the intersection (that is, a wide pullback) of full subcategories as above. Therefore the forgetful $\cv$-functor $U\colon \Mod(\mathbb E)\to \cv$ is continuous and $\lambda$-ary, and hence it has a left adjoint (since the domain is locally presentable). Finally, it is easy to see that $U$ strictly creates coequalizers for $U$-absolute pairs, making it strictly monadic.
\end{proof}

Conversely, we can describe $\cv$-categories of algebras as given by models of equational theories. We stress out that the equational theory that we construct below only involves the terms of Definition~\ref{terms}, not the extended ones. Therefore, this syntactically improves the result of \cite[5.24]{LP} where extended terms are used, and justifies our choice of terms.

\begin{propo}\label{monad->equations}
	Let $T\colon \cv\to\cv$ be a $\lambda$-ary monad. Then there exists a $\lambda$-ary equational theory $\mathbb E$ on a $\lambda$-ary language $\mathbb L$ together with an isomorphism $E\colon \tx{Alg}(T)\to \Mod(\mathbb E)$ making the triangle
	\begin{center}
		\begin{tikzpicture}[baseline=(current  bounding  box.south), scale=2]
			
			\node (a0) at (0,0.8) {$\tx{Alg}(T)$};
			\node (b0) at (1.3,0.8) {$\Mod(\mathbb E)$};
			\node (c0) at (0.65,0) {$\cv$};
			
			\path[font=\scriptsize]
			
			(a0) edge [->] node [above] {$E$} (b0)
			(a0) edge [->] node [left] {$U$} (c0)
			(b0) edge [->] node [right] {$U'$} (c0);
		\end{tikzpicture}	
	\end{center} 
	commute.
\end{propo}
\begin{proof}
	Let $T\colon \cv\to\cv$ be a $\lambda$-ary monad. We need to find a $\lambda$-ary equational theory $\mathbb E$ on a $\lambda$-ary language $\mathbb L$ together with an isomorphism $E\colon \tx{Alg}(T)\to \Mod(\mathbb E)$ that respects the forgetful $\cv$-functors.
	
	By \cite[2.4]{BG} we can find a $\cv_\lambda$-theory $H\colon\cv_\lambda^{\op}\to \ct$ for which $\tx{Alg}(T)$ is given by the pullback below.
	\begin{center}
		\begin{tikzpicture}[baseline=(current  bounding  box.south), scale=2]
			
			\node (a0) at (0,0.8) {$\tx{Alg}(T)$};
			\node (a0') at (0.3,0.6) {$\lrcorner$};
			\node (b0) at (1.3,0.8) {$[\ct,\cv]$};
			\node (c0) at (0,0) {$\cv$};
			\node (d0) at (1.3,0) {$[\cv_\lambda^{op},\cv]$};
			
			\path[font=\scriptsize]
			
			(a0) edge [right hook->] node [above] {} (b0)
			(a0) edge [->] node [left] {$U$} (c0)
			(b0) edge [->] node [right] {$[H,\cv]$} (d0)
			(c0) edge [right hook->] node [below] {$\cv(K-,1)$} (d0);
		\end{tikzpicture}	
	\end{center} 
	Such $H$ can be chosen to be the left part of the (identity on objects, fully faithful) factorization of 
	$$ \cv_\lambda^{\op}\xrightarrow{K^{\op}}\cv^{\op}\xrightarrow{F^{\op}} \tx{Alg}(T)^{\op}$$
	where $F$ is the left adjoint to $U$. In particular then we can assume that $\ct$ has and $H$ preserves $\lambda$-small powers, so that for any $X,Y\in\cv_\lambda$ the power of $Y$ by $X$ in $\ct$ is simply the (image through $H$ of the) tensor product $X\otimes Y$. 
	Under these assumptions on $H$, a $\cv$-functor $G\colon\ct\to \cv$ preserves $\lambda$-small powers if and only if $GH$ does. Therefore, since $\cv(K-,X)$ preserves $\lambda$-small powers for any $X\in\cv$, then $\tx{Alg}(T)$ is also defined by the pullback 
	\begin{center}
		\begin{tikzpicture}[baseline=(current  bounding  box.south), scale=2]
			
			\node (a0) at (0,0.8) {$\tx{Alg}(T)$};
			\node (a0') at (0.3,0.6) {$\lrcorner$};
			\node (b0) at (1.5,0.8) {$\lambda\tx{-Pw}[\ct,\cv]$};
			\node (c0) at (0,0) {$\cv$};
			\node (d0) at (1.5,0) {$\lambda\tx{-Pw}[\cv_\lambda^{op},\cv]$};
			
			\path[font=\scriptsize]
			
			(a0) edge [right hook->] node [above] {} (b0)
			(a0) edge [->] node [left] {$U$} (c0)
			(b0) edge [->] node [right] {$[H,\cv]$} (d0)
			(c0) edge [right hook->] node [below] {$\cv(K-,1)$} (d0);
		\end{tikzpicture}	
	\end{center} 
	where $\lambda\tx{-Pw}[\ca,\cv]$ is the full subcategory of $[\ca,\cv]$ spanned by those $\cv$-functors that preserve $\lambda$-small powers. To conclude it is enough to construct a language $\mathbb L$ and an equational $\mathbb L$-theory $\mathbb E$ for which $\Mod(\mathbb E)$ is also presented as the pullback above.
	
	Consider the language $\mathbb L$ defined by a function symbol $\overline f:(X,Y)$ for any morphism $f\colon X\to Y$ in $\ct$. Note that for any $g\colon Y\to X$ in $\cv_\lambda$ we have two different $(X,Y)$-ary terms given by $g$ (from the first rule) and $\overline{H(g)}$ (from the second).
	The $\mathbb L$-theory $\mathbb E$ is given by the following equations:\begin{enumerate}
		\item[(a)] $\overline f( \overline g)=\overline{fg}$, for any composable maps $f,g$ in $\ct$;
		\item[(b)] $\overline{1_X} (\overline g)=\overline g$ and $\overline g( \overline{1_Y})=\overline g$, for any $g\colon X\to Y$ in $\ct$;
		\item[(c)] $g=\overline{H(g)}$, for any morphism $g$ in $\cv_\lambda$;
		\item[(d)] $\overline{Z\otimes f} = \overline f^Z$ for any morphism $f\colon X\to Y$ in $\ct$ and $Z\in\cv_\lambda$; here $Z\otimes f\colon Z\otimes X\to Z\otimes Y$ is the power of $f$ by $Z$ in $\ct$.
	\end{enumerate}
	
	Now, to give an $\mathbb L$-structure satisfying axioms (a) and (b) is the same as giving an object $A$ of $\cv$ together with an ordinary functor $\tilde{A}\colon \ct_0\to\cv_0$ for which $\tilde{A}(X)=A^X$. Axiom (c) says that $\tilde{A}\circ H_0=A^{(-)}\colon (\cv_\lambda)_0^{op}\to\cv_0$. Finally, axiom (d) says that for any object $Z$ in $\cv_\lambda$ and morphism $f\colon X\to Y$ in $\ct$, the square below commutes,
	\begin{center}
		\begin{tikzpicture}[baseline=(current  bounding  box.south), scale=2]
			
			\node (a0) at (0,0.8) {$A^{Z\otimes X}$};
			\node (b0) at (1.4,0.8) {$A^{Z\otimes Y}$};
			\node (c0) at (0,0) {$(A^X)^Z$};
			\node (d0) at (1.4,0) {$(A^Y)^Z$};
			
			\path[font=\scriptsize]
			
			(a0) edge [->] node [above] {$\tilde{A}(Z\otimes f)$} (b0)
			(a0) edge [->] node [left] {$\cong$} (c0)
			(b0) edge [->] node [right] {$\cong$} (d0)
			(c0) edge [->] node [below] {$\tilde{A}(f)^Z$} (d0);
		\end{tikzpicture}	
	\end{center}
	where the vertical maps are the natural comparison isomorphisms. Since the analogous commutativity property holds in the first variable (that is, $\tilde{A}(h\otimes X)\cong \tilde{A}(X)^h$ for $h\in\cv_\lambda$) because $\tilde{A}$ restricts to $A^{(-)}$, this means that $\tilde{A}$ preserves the action defined by taking $\lambda$-small powers, up to coherent natural isomorphism. Therefore, by the infinitary version \cite[9.2]{LT20} (see also \cite{JK}), to give a model of $\mathbb E$ is the same as giving a $\cv$-functor $\hat{A}\colon\ct\to\cv$ which preserves $\lambda$-small powers and that restricts to $A^{(-)}\colon\cv_\lambda^{op}\to\cv$. Similarly, since a morphism of $\mathbb L$-structures is determined by a map $h\colon A\to B$ in $\cv$, if $A$ and $B$ are models of $\mathbb E$ then $h$ induces an ordinary natural transformation $\eta\colon \tilde{A}\to \tilde{B}$ defined by $\eta_X=h^X$. This transformation is clearly compatible with the action given by $\lambda$-small powers; thus (again by \cite[9.2]{LT20}) it is actually a $\cv$-natural transformation $\bar \eta\colon \hat{A}\to \hat {B}$. As a consequence $\Mod(\mathbb E)_0$ is a pullback
	\begin{center}
		\begin{tikzpicture}[baseline=(current  bounding  box.south), scale=2]
			
			\node (a0) at (0,0.8) {$\Mod(\mathbb E)_0$};
			\node (a0') at (0.3,0.6) {$\lrcorner$};
			\node (b0) at (1.55,0.8) {$\lambda\tx{-Pw}[\ct,\cv]_0$};
			\node (c0) at (0,0) {$\cv$};
			\node (d0) at (1.55,0) {$\lambda\tx{-Pw}[\cv_\lambda^{op},\cv]_0$};
			
			\path[font=\scriptsize]
			
			(a0) edge [right hook->] node [above] {} (b0)
			(a0) edge [->] node [left] {$U'_0$} (c0)
			(b0) edge [->] node [right] {$[H,\cv]_0$} (d0)
			(c0) edge [right hook->] node [below] {$\cv(K-,1)_0$} (d0);
		\end{tikzpicture}	
	\end{center} 
	of ordinary categories. Since the top horizontal arrow preserves the action given by taking powers (these are defined pointwise both in $\Mod(\mathbb E)$ and $\lambda\tx{-Pw}[\ct,\cv]$), then it extends to an actual $\cv$-functor preserving powers. Finally, since all the other $\cv$-functors involved preserve such powers, then the pullback of ordinary categories is actually a pullback of enriched categories.
\end{proof}

Below, we will say that an object $G$ in a $\cv$-category $\ck$ is {\em $\cv$-projective} if $\ck(G,-)$ preserves coequalizers of $\ck(G,-)$-split pairs.

We now put together the results above in the following characterization theorem. The real improvement of this theorem, with rest to results already known in the literature, is that it is enough to consider the recursively generated terms of Definition~\ref{terms}, rather than having to go all the way to the less satisfactory notion of extended term.

\begin{theo}\label{char-single}
	The following are equivalent for a $\cv$-category $\ck$: \begin{enumerate}
		\item $\ck\simeq\Mod(\mathbb E)$ for an extended $\lambda$-ary equational theory $\mathbb E$;
		\item $\ck\simeq\Mod(\mathbb E)$ for a $\lambda$-ary equational theory $\mathbb E$;
		\item $\ck\simeq\tx{Alg}(T)$ for a $\lambda$-ary monad $T$ on $\cv$;
		\item $\ck$ is cocomplete and has a $\lambda$-presentable and $\cv$-projective strong generator $G\in\ck$;
		\item $\ck\simeq\lambda\tx{-Pw}(\ct,\cv)$ is equivalent to the $\cv$-category of $\cv$-functors preserving $\lambda$-small powers, for some $\cv_\lambda$-theory $\ct$.
	\end{enumerate}
\end{theo}

The equivalence $(5)\Leftrightarrow(3)$ was first shown in the finitary setting by Power~\cite{Po}; the infinitary version follows from the monad theory correspondence of Bourke and Garner~\cite{BG}. While $(1)\Leftrightarrow(3)$ was shown as \cite[5.26]{LP} and the unenriched version of $(1)\Leftrightarrow(4)$ appears in~\cite{AR}.

\begin{proof}
	$(3)\Rightarrow(2)$ is Proposition~\ref{monad->equations} and $(2)\Rightarrow(1)$ is trivial. For $(1)\Rightarrow(4)$, note that $\ck$ is locally $\lambda$-presentable and the forgetful $U_\mathbb E\colon \ck\simeq\Mod(\mathbb E)\to \cv$ is continuous, $\lambda$-ary. Thus $U_\mathbb E$ has a left adjoint $L$ whose value at $I$ gives an object $G:=LI\in\ck$ for which $U_\mathbb E\cong\ck(G,-)$. Since $U_\mathbb E$ is conservative, $\lambda$-ary, and preserves $U_\mathbb E$-split coequalizers (being monadic), it follows that $G$ has the desired properties.
	
	$(4)\Rightarrow(3)$. Note that the $\cv$-category $\ck$ is locally $\lambda$-presentable and that $$U_\ck:=\ck(G,-)\colon\ck\to\cv$$ is (by hypothesis) continuous, $\lambda$-ary, and preserves coequalizers of $U$-split pairs. Thus $U_\ck$ has a left adjoint and is $\lambda$-ary monadic by the monadicity theorem.
	
	$(5)\Leftrightarrow(3)$. This is given by (the infinitary version of) \cite[Example~44.(vi)]{BG}.
\end{proof}

\begin{rem}\label{products}
	As it was already explained in~\cite{Po}, in the enriched context we need to ask for preservation of $\lambda$-small powers, instead of $\lambda$-small products. This is because the $\cv$-functor $\cv(K,1)\colon \cv\to [\cv_\lambda^{\op},\cv]$ restricts to an equivalence
	$$\cv\simeq\lambda\tx{-Pw}(\cv_\lambda^{\op},\cv)$$
	whose inverse is obtained by sending $F\colon \cv_\lambda^{\op}\to\cv$ to $F(I)$. Note, however, that every $F\colon \cv_\lambda^{\op}\to \cv$ preserving $\lambda$-small powers also preserves $\lambda$-small products; indeed, every such $F$ is of the form $F(X)=A^X$, and this preserves $\lambda$-small products.
	
	As a consequence, for any $\cv_\lambda$-theory $\tau\colon\cv_\lambda^{\op}\to \ct$, every $\cv$-functor $\ct\to\cv$ preserving $\lambda$-small powers also preserves $\lambda$-small products (since, by \cite{BG}, $\tau$ always preserves all $\lambda$-small limits in $\cv_\lambda^{op}$).
\end{rem}

\begin{exams}\label{ex}
	{\setlength{\leftmargini}{1.6em}
		\begin{enumerate}
			\item 	Over $\cv=\bo{Pos}$, our models of equational theories correspond to the varieties of ordered (coherent) algebras of \cite{AFMS}. This can be seen as a consequence of the characterizations above (that they also obtain with their language); however, there is a much deeper correlation between our approaches (see \cite[4.11(1)]{R3}).
			
			Following Example~\ref{extended}(2), a signature $\mathbb L$ in context in \cite[3.2]{AFMS} is the same as our finitary language with function symbols of arity $(X,1)$ and coherent $\mathbb L$-algebras of \cite{AFMS} are $\mathbb L$-structures in our sense. With regards to formulas, in \cite[3.15]{AFMS} one is allowed to consider {\em inequations} of the form $(s\leq t)$ for $s,t$ of arity $(X,1)$. These can be interpreted as equations in our language by adding a new function symbol $q$ of arity $(X,\bf 2)$ --- where $\bf 2$ is the two element chain $\{0\to 1\}$. Indeed, their $(s\leq t)$ is then equivalent to our equations
			$$ (i_1( q)=s)\ \text{ and }\ (i_2( q)=t)$$
			where $i_0,i_1\colon 1\to \bf 2$ are the two inclusions.
			
			Conversely, given a language $\mathbb L$ over $\Pos$, to interpret an $(X,Y)$-ary symbol $f$ from $\mathbb L$, is the same as to have the interpretation of a family $f_y:(X,1)$, for $y\in Y$, satisfying inequations $f_y\leq f_y'$ for any $y\leq y'$ in $Y$. Thus, every $\bo{Pos}$-category $\Mod(\mathbb E)$ of $\mathbb E$-models has a clear interpretation as a variety of ordered algebras.
			Finally, within our terms we are allowed to take powers by an arity $Z$; since $\bo{Pos}(1,-)$ is faithful, these can be avoided by Corollary~\ref{elimin-powers}.  
			
			\item Over $\cv=\Met$, our models of $\aleph_1$-ary equational theories correspond to the $\omega_1$-varieties of quantitative algebras of \cite{MPP,MPP1,A}. This correlation was observed in \cite[4.11(2)]{R3} and, again, it follows from Corollary~\ref{elimin-powers}. Recall that $\Met$ is only locally $\aleph_1$-presentable.
			One proceeds like in (1) but, instead of inequations, one has \textit{quantitative equations} $(s=_\eps t)$ where $\eps> 0$. These can be seen as equations using a function symbol $q$ of arity $(X,2_\eps)$ where $2_\eps$ is from Example~\ref{kant}. Indeed, $(s=_\eps t)$ is then equivalent to our equations
			$$ (i_1 (q)=s)\ \text{ and }\ (i_2 ( q)=t)$$
			where $i_0,i_1\colon 1\to 2_\eps$ are the two inclusions.
			
			Conversely, given a language $\mathbb L$ over $\Met$, to interpret an $(X,Y)$-ary symbol $f$ from $\mathbb L$, is the same as to have the interpretation of a family $f_y:(X,1)$, for $y\in Y$, satisfying quantitative equations $f_y=_\eps f_y'$ for any $y,y'\in Y$ such that $d(y,y')\leq\eps$. A concrete example is Example~\ref{kant}. 
			Again, since $1$ is a generator, power terms can be avoided by Corollary~\ref{elimin-powers}. 
			
			\item Over $\cv=\omega$-$\CPO$, our models of equational theories include the varieties of continuous algebras of \cite{ANR}. In \cite{ANR,ADV}, a signature $\Sigma$ is the same as our language with function symbols of arity $(X,1)$ where $X$ is a countable antichain. Except standard terms, they allow countable joins $\bigvee_{n<\omega} t_i$ of terms. But their interpretation is tailored such that $t_0\leq t_1\leq\cdots\leq t_n\leq\cdots$. We express $s\leq t$ in the same way as in (1): that is, by adding new function symbols of arity $(X,\bf 2)$. To express joins of countable terms, we add a new function symbol $q$ of arity $(X,\omega +1)$. Then $t=\bigvee_{n<\omega}t_n$ is equivalent to our equations
			$$ (i_n (q)=t_n), \text{\ for\ } n<\omega, \text{ and } (i_\omega( q)=t)$$
			where $i_n\colon 1\to \omega+1$ correspond to $n\leq\omega$.
			
			Conversely, let $\mathbb L$ be a language whose function symbols have countable antichains as the input arities. Then its function symbols of arities $(X,1)$ form a signature from \cite{ANR,ADV}. An $(X,Y)$-ary function symbol $f$ from $\mathbb L$ is interpreted as the family $f_y:(X,1)$, for $y\in Y$, satisfying 
			$$
			f_{\vee y_i}=\bigvee f_{y_i}
			$$
			for $y_0\leq y_1\leq\cdots y_n\leq\cdots$.
			Since $1$ is a generator, power terms can be avoided by Corollary~\ref{elimin-powers}.  
		\end{enumerate}
	}
\end{exams}

\begin{exam}
	The $\Met$-category $\Ban$ of Banach spaces is $\aleph_1$-ary monadic over $\Met$ by \cite{R2}; hence is the $\Met$-category of models of a $\aleph_1$-ary theory $\mathbb E$ over an $\aleph_1$-language $\mathbb L$ by Proposition~\ref{monad->equations}. We do not know whether there is a nice choice of $\mathbb L$ and $\mathbb E$ that presents Banach spaces.
\end{exam}

To conclude this section we characterize $\cv$-categories of models of equational theories as certain enriched orthogonality classes; this will be useful for the Birkhoff variety theorems. 

An object $X$ of a $\cv$-category $\ck$ is said to be {\em orthogonal with respect to $h\colon A\to B$} if the map $$\ck(h,X)\colon \ck(B,X)\to\ck(A,X)$$ is an isomorphism in $\cv$. A full subcategory of $\ck$ spanned by objects orthogonal with respect to a collection of maps is called an {\em orthogonality class}. Then:

\begin{propo}\label{regular-epi-orth}
	Let $\mathbb L$ be a $\lambda$-ary language, and $\lambda\leq\kappa\leq\infty$. Then classes defined by extended $\kappa$-ary equational $\mathbb L$-theories in $\Str(\mathbb L)$ are precisely given by orthogonality classes defined with respect to maps of the form 
	$$ h\colon FX\twoheadrightarrow W $$
	in $\Str(\mathbb L)$, where $X\in\cv_\kappa$ and $h$ is a regular epimorphism.
\end{propo}
\begin{proof}
	On one hand, if we are given an equation $(s=t)$, with extended terms $s,t\colon FY\to FX$ in $\Str(\mathbb L)$, we can consider the coequalizer $h\colon FX\to W$ of $(s,t)$. It follows that an $A\in \Str(\mathbb L)$ satisfies $(s=t)$ if and only if $\Str(\mathbb L)(s,A)=\Str(\mathbb L)(t,A)$, if and only if the equalizer of the pair $(\Str(\mathbb L)(s,A),\Str(\mathbb L)(t,A))$ in $\cv$ is an isomorphism. But that equalizer is exactly $\Str(\mathbb L)(h,A)$. Thus $A$ satisfies $(s=t)$ if and only if it is orthogonal with respect to $h$.
	
	Conversely, given a regular epimorphism $ h\colon FX\twoheadrightarrow W $ with $X\in\cv_\kappa$, we can consider its kernel pair $(s',t')\colon K\to FX$, and find an epimorphic family of maps $\{m_i\colon FX_i\to K\}_{i\in J}$ with $X_i\in\cv_\kappa$ (since these form a strong generator). Let now $s_i:=s'm_i$ and $t_i:=t'm_i$. Arguing as above it follows that, given $A\in \Str(\mathbb L)$, the arrow $\Str(\mathbb L)(h,A)$ is an isomorphism if and only if $\Str(\mathbb L)(s',A)=\Str(\mathbb L)(t',A)$, if and only if $\Str(\mathbb L)(s_i,A)=\Str(\mathbb L)(t_i,A)$ for all $i\in J$, if and only if $A$ satisfies $(s_i=t_i)$ for any $i\in J$.
\end{proof}

\subsection{Elimination of arities}\label{elimination-arities}$ $

Now we turn to the elimination of arities and of extended terms. The next corollary shows that extended terms can always be replaced by standard ones, at the cost of changing language. This is slightly stronger than the implication $(1)\Rightarrow (2)$ of Theorem~\ref{char-single} since we talk about an isomorphism rather than an equivalence.

\begin{coro}\label{ext->stand}
	Every $\cv$-category of models $\Mod(\mathbb E)$ of an extended $\lambda$-ary equational theory is isomorphic, as a $\cv$-category over $\cv$, to $\Mod(\mathbb E')$ where $\mathbb E'$ is a $\lambda$-ary equational $\mathbb L'$-theory.
\end{coro}
\begin{proof}
	Follows from putting together Propositions~\ref{monad->equations} and~\ref{eqations->monad}.
\end{proof}

Recall that a set of objects $\cg$ of $\cv_0$ is called a {\em generator} if the functors $\cv_0(G,-)$, for $G\in\cg$, are jointly faithful. Then we prove:

\begin{propo}\label{G-powers}
	Let $\cg\subseteq\cv_\lambda$ be a generator of $\cv_0$. Every $\cv$-category of models $\Mod(\mathbb E)$ of a $\lambda$-ary equational $\mathbb L$-theory is isomorphic, as a $\cv$-category over $\cv$, to $\Mod(\mathbb E')$ where $\mathbb E'$ is a $\lambda$-ary equational $\mathbb L'$-theory, over some other language $\mathbb L'$, with terms obtained by restricting rule (3) of \ref{terms} only to $Z\in\cg$.
\end{propo}
\begin{proof}
	Thanks to the proof of Proposition~\ref{monad->equations} we can assume that the formulas using power terms are all of the form $(t^Z=s)$ where $s$ has arity $(Z\otimes X,Z\otimes Y)$ and $t$ has arity $(X,Y)$.
	To conclude, it is enough to prove that the equality $t^Z=s$ holds in an $\mathbb L$-structure $A$ if and only if $$(t^G(z\otimes X)=(z\otimes Y)(s))$$ holds in $A$ for any $z\colon G\to Z$ with $G\in\cg$, where $z\otimes X$ is the term corresponding to the map
	$$z\otimes X\colon G\otimes X\to Z\otimes X$$ 
	in $\cv_\lambda$. Now note that the equality $t^Z_A=s_A$ holds if and only if the solid square below
	\begin{center}
		\begin{tikzpicture}[baseline=(current  bounding  box.south), scale=2]
			
			\node (0) at (-0.9,0.8) {$G$};
			\node (a0) at (0,0.8) {$Z$};
			\node (b0) at (2,0.8) {$[A^{Z\otimes Y},A^Y]$};
			\node (c0) at (0,0) {$[A^{Z\otimes X},A^X]$};
			\node (d0) at (2,0) {$[A^{Z\otimes X},A^Y]$};
			
			\path[font=\scriptsize]
			
			(0) edge [dashed, ->] node [above] {$z$} (a0)
			(a0) edge [->] node [above] {$1_{Z\otimes Y}'$} (b0)
			(a0) edge [->] node [left] {$1_{Z\otimes X}'$} (c0)
			(b0) edge [->] node [right] {$[s,A^Y]$} (d0)
			(c0) edge [->] node [below] {$[A^{Z\otimes X},t]$} (d0);
		\end{tikzpicture}	
	\end{center}
	commutes in $\cv$, where $1_{Z\otimes X}'$ and $1_{Z\otimes Y}'$ are the transposes of the identities. By the hypothesis on $\cg$, this square commutes if and only if the squares obtained after pre-composing with all maps $z\colon G\to Z$, for $G\in\cg$, commute. It is now easy to see that, for each such $z$, the lower composite of the new square transposes to $(t^G(z\otimes X))_A$ and the upper one to $((z\otimes Y)(s))_A$. Thus $t^Z_A=s_A$ holds if and only if $(t^G(z\otimes X))_A=((z\otimes Y)(s))_A$ holds for any $z$ as above.
\end{proof}

As a direct consequence:

\begin{coro}\label{elimin-powers}
	Let the unit $I$ be a generator in $\cv_0$. Every $\cv$-category of models $\Mod(\mathbb E)$ of a $\lambda$-ary equational theory on $\mathbb L$ is isomorphic, as a $\cv$-category over $\cv$, to $\Mod(\mathbb E')$ where $\mathbb E'$ is a $\lambda$-ary equational $\mathbb L'$-theory involving terms obtained by applying the rules (1), (2), and (4) of \ref{terms}. 
\end{coro}
\begin{proof}
	By Proposition~\ref{G-powers} above we need to apply rule (3) only for $Z=I$, making it trivial. 
\end{proof}

Next we show that equations can be modified so that the output arities are restricted to a generator of $\cv_0$. Note that the language $\mathbb L$ itself can have output arities not in $\cg$. 

\begin{propo}\label{G-output}
	Let $\cg\subseteq\cv_\lambda$ be a generator of $\cv_0$. For every $\lambda$-ary equational $\mathbb L$-theory $\mathbb E$ there exists a $\lambda$-ary equational $\mathbb L$-theory $\mathbb E'$ such that  \begin{itemize}
		\item $\Mod(\mathbb E)$ is isomorphic, as a $\cv$-category over $\cv$, to $\Mod(\mathbb E')$;
		\item the equations of $\mathbb E'$ are of the form $(s=t)$ where the output arity, common to $s$ and $t$, lies in $\cg$.
	\end{itemize}
\end{propo}
\begin{proof}
	Consider an equation $(s=t)$ in $\mathbb E$ with arities $(X,Y)$, and let $h_i\colon G_i\to Y$ be an epimorphic family over $Y$ with $G_i\in\cg$ for any $i$. Then for any $\mathbb L$-structure $A$ the family $\{A^{h_i}\colon A^Y\to A^{G_i}\}_i$ is jointly monomorphic; it follows that $A$ satisfies $(s=t)$ if and only if it satisfies $(h_i(s)=h_i(t))$ for any $i$. 
\end{proof}

We conclude this section with the following theorem which will be useful when trying to express in simple terms what our enriched universal algebra looks like for specific instances of enrichment; see Remark~\ref{applications}. This is a specialization of \cite[3.17]{R3} to our setting.

\begin{theo}\label{elimination}
	Let $\cg\subseteq\cv_\lambda$ be a generator of $\cv_0$. Every $\cv$-category of models $\Mod(\mathbb E)$ of an extended $\lambda$-ary equational $\mathbb L$-theory is isomorphic, as a $\cv$-category over $\cv$, to $\Mod(\mathbb E')$ where $\mathbb E'$ is a $\lambda$-ary equational $\mathbb L'$-theory whose terms:
		\begin{enumerate}\setlength\itemsep{0.05em}
			\item have output arity in $\cg$;
			\item are obtained by restricting rule (3) of \ref{terms} only to $Z\in\cg$.
		\end{enumerate}
\end{theo}
\begin{proof}
	Follows from Corollary~\ref{ext->stand} and Propositions~\ref{G-powers} and~\ref{G-output}.
\end{proof}

\begin{rem}\label{applications}
	When $\cv=\Set,\Pos,\Met,$ and $\omega$-$\CPO$ we can choose $\cg=\{1\}$. Thus, with Theorem~\ref{elimination} we recover what we had calculated explicitly in Examples~\ref{ex} using results from the literature: for the bases of enrichment mentioned above it is enough to consider $(X,1)$-ary terms. When $\cv=\Cat$ we can consider $\cg=\{\bf{2}\}$, while for $\cv=\SSet$ we can choose $\cg=\bf{\Delta}$. This is potentially useful to develop explicitly 2-dimensional and simplicial universal algebra.
\end{rem}

\section{Enriched Birkhoff subcategories}\label{birkhoff-section}

We introduce the notion of enriched Birkhoff subcategory, which is supposed to characterize those full subcategories of $\Str(\mathbb L)$ that are of the form $\Mod(\mathbb E)$, for some $\mathbb L$-theory $\mathbb E$. An unenriched Birkhoff-type theorem for a general ambient category $\ck$ (in place of $\Str(\mathbb L)$) was proved by Manes \cite{M}; providing a starting point for the proof of our result. However, his ``equations'' were just quotients of free algebras $FX$; this approach was made enriched in \cite{MU}. To transform these ``equations'' to real equations of extended terms we shall need more hypotheses on $\cv$. In a special case, this was done in \cite{R1}.

First of all, we need to show what closure properties $\Mod(\mathbb E)$ does satisfy. Throughout this section $U\colon \Str(\mathbb L)\to\cv$ is the forgetful $\cv$-functor. Below we say that a map $f\colon A\to B$  of $\mathbb L$-structures is $\cv$-\textit{split} if it is a split epimorphism in $\cv$; equivalently, if it is $U$-split. Note that then $f$ is necessarily an epimorphism in $\Str(\mathbb L)$. By a {\em substructure} of $B$ in $\Str(\mathbb L)$ we mean an object $A$ together with a monomorphism $A\rightarrowtail B$ in $\Str(\mathbb L)$.

\begin{propo}\label{Mod->Birk}
	Let $\mathbb L$ be a $\lambda$-ary language and $\mathbb E$ an extended equational theory on $\mathbb L$. Then $\Mod(\mathbb E)$ is closed in $\Str(\mathbb L)$ under products, powers, substructures, and $\cv$-split quotients. If $\mathbb E$ is $\lambda$-ary, then $\Mod(\mathbb E)$ is also closed under $\lambda$-filtered colimits.
\end{propo}
\begin{proof}
	The closure under products is evident since $$s_{\prod A_i}\cong \textstyle\prod s_{A_i}$$ for any extended term $s$ and $\mathbb L$-structures $A_i$. Argue in the same manner for closure under $\lambda$-filtered colimits when $\mathbb E$ is $\lambda$-ary.
	
	Assume now that $A$ satisfies the equation $(s=t)$ where $s$ and $t$ are $(X,Y)$-ary. Then $A$ satisfies the equation $s^Z=t^Z$ for every $Z$ in $\cv$. Thus, since the interpretation of $t$ at $A^Z$ below
	$$ t_{A^Z}\colon (A^Z)^X\to (A^Z)^Y $$
	is naturally isomorphic to
	$$ t_A^Z\colon A^{Z\otimes X}\to A^{Z\otimes Y} $$
	(and the same for $s$), it follows that $A^Z$ also satisfies $(s=t)$.
	
	Let now $f\colon A\to B$ be a monomorphism of $\mathbb L$-structures such that $B$ is an $\mathbb E$-model. For any equation $(s=t)$ in $\mathbb E$ we can consider the diagram below
	\begin{center}
		\begin{tikzpicture}[baseline=(current  bounding  box.south), scale=2]
			
			\node (z) at (-1,0) {$A^X$};
			\node (a) at (0,0) {$A^Y$};
			\node (b) at (-1,-0.8) {$B^X$};
			\node (b1) at (0,-0.8) {$B^Y$};

			\path[font=\scriptsize]
			
			(z) edge [->] node [left] {$f^X$} (b)
			(a) edge [->] node [right] {$f^Y$} (b1)
			
			([yshift=-1.5pt]z.east) edge [->] node [below] {$s_A$} ([yshift=-1.5pt]a.west)
			([yshift=1.5pt]z.east) edge [->] node [above] {$t_A$} ([yshift=1.5pt]a.west)
			([yshift=-1.5pt]b.east) edge [->] node [below] {$s_B$} ([yshift=-1.5pt]b1.west)
			([yshift=1.5pt]b.east) edge [->] node [above] {$t_B$} ([yshift=1.5pt]b1.west);
		\end{tikzpicture}
	\end{center}
	where $f^Y$ is a monomorphism in $\cv$. Thus $s_B=t_B$ implies that $s_A=t_A$, and hence $A$ is a model of $\mathbb E$. 
	Finally, if $f\colon A\to B$ is a morphism of $\mathbb L$-structures which splits in $\cv$ and such that $A$ is an $\mathbb E$-model, then, in the diagram above, $f^X$ is a split epimorphism in $\cv$ (since $f$ was). Thus $s_A=t_A$ implies that $s_B=t_B$, and hence $B$ is a model of $\mathbb E$.
\end{proof}

Thus we define:

\begin{defi}
	We say that $\cl$ is an \textit{enriched Birkhoff subcategory} of $\Str(\mathbb L)$ if it is a full replete subcategory of $\Str(\mathbb L)$ closed under:\begin{itemize}
		\item products and powers;
		\item substructures: if $A\rightarrowtail B$ is a monomorphism in $\Str(\mathbb L)$ and $B\in\cl$, then $A\in\cl$;
		\item $\cv$-split quotients: if $A\twoheadrightarrow B$ in $\Str(\mathbb L)$ is an epimorphism that splits in $\cv$ and $A\in\cl$, then $B\in\cl$.
	\end{itemize}
\end{defi}
 
This is an enriched version of the notion considered in \cite{M}, where Birkhoff subcategories are defined as being closed under products, substructures, and $U$-split quotients. 

\begin{rem}
	If $I$ is a generator in $\cv_0$ then every Birkhoff subcategory of $\Str(\mathbb L)_0$ is enriched (see \cite[5.2]{R1}).
\end{rem}

\begin{exam}
	Consider the full subcategory $\cl$ of $\Gra$ consisting of the graphs with no edges and the unit $I$ (given by the graph with one vertex and one edge). Then $\cl$ is an unenriched Birkhoff subcategory of $\Gra$ but not an enriched one because it is not closed under powers by $1$. Here, we consider $\Gra$ as $\Str(\emptyset)$. Clearly, $\cl$ is given by the unenriched satisfaction of equations $p_1=p_2$ where
	$p_1,p_2\colon G\times G\to G$ with $G$ non-trivial. But $\cl$ does not satisfy
	$p_1^1=p_2^1$ because $p_1^1,p_2^1\colon G_t\times G_t\to G_t$ where $G_t$ is the trivial graph having the same vertices as $G$.
\end{exam}

We shall now characterize enriched Birkhoff subcategories via $\infty$-ary equational theories. Note that the $\cv$-category $\ct_\mathbb L^\infty$ of $\infty$-ary extended terms is not small in general, since it is partly generated by all morphisms in $\cv$. Thus, $\infty$-ary equational theories will often involve a large class of equations.

\begin{theo}\label{birkhoff}
	Let $\mathbb L$ be a $\lambda$-ary language for which in $\Str(\mathbb L)$ every strong epimorphism is regular. Then enriched Birkhoff subcategories of $\Str(\mathbb L)$ are precisely classes defined by extended $\infty$-ary equational $\mathbb L$-theories.
\end{theo}
\begin{proof}
	One direction is given by Proposition~\ref{Mod->Birk}. For the other, let $R\colon \cl\hookrightarrow\Str(\mathbb L)$ be an enriched Birkhoff subcategory. By \cite[Chapter 3, 3.4]{M} applied to $T=UF$, we know that $\cl_0$ is reflective in $\Str(\mathbb L)_0$ (as ordinary categories), with left adjoint $L$, and that the units $\rho_Z\colon FZ\twoheadrightarrow F'Z=RLFZ$ become epimorphism in $\cv$ once we apply $U$ (in the notation of \cite{M} we have $T=UF$ and $T'=URLF$). Now, since $U$ is faithful, it reflects epimorphisms, so that each $\rho_Z$ is an epimorphism in $\Str(\mathbb L)$. But $\Str(\mathbb L)$ has the (strong epi, mono) factorization (by \cite[1.61]{AR}); therefore, using that $\cl$ is closed under substructures, it is easy to see that $\rho_Z$ is necessarily a strong epimorphisms in $\Str(\mathbb L)$.
	 
	Now note that $\cl_0$ is defined by the orthogonality class with respect to the maps $\rho_Z$ for any $Z\in\cv$; that is, $L\in\cl$ if and only if $\Str(\mathbb L)_0(\rho_Z,L)$ is a bijection for any $Z\in\cv$. Indeed, this is essentially a rephrasing of \cite[Chapter 3, 3.3]{M}.
	
	Then, since $\cl$ is closed under powers, $\cl$ is actually the enriched orthogonality class induced by the maps $\rho_Z$, for $Z\in\cv$. By hypothesis on $\Str(\mathbb L)$ each of these maps is a regular epimorphism; thus $\cl$ is defined by an $\infty$-ary equational $\mathbb L$-theory by Proposition~\ref{regular-epi-orth}.
\end{proof}

Next we focus on $\kappa$-ary equational theories for $\lambda\leq\kappa\leq\infty$.

\begin{defi}\label{l-birkhoff}
	Let $\mathbb L$ be a $\lambda$-ary language and $\lambda\leq\kappa\leq\infty$. We say that $\cl$ is an \textit{enriched $\kappa$-Birkhoff subcategory} of $\Str(\mathbb L)$ if it is closed under products, powers, substructures, $\cv$-split quotients, and $\kappa$-directed colimits.
\end{defi}

For $\kappa=\infty$ we recover the notion of enriched Birkhoff subcategory since $\infty$-directed colimits reduce to coequalizers of split pairs, and these are already captured by closure under $\cv$-split quotients.

\begin{theo}\label{l-birkhoff1}
	Let $\mathbb L$ be a $\lambda$-ary language for which in $\Str(\mathbb L)$ every strong epimorphism is regular, and let $\lambda\leq\kappa\leq\infty$. Then enriched $\kappa$-Birkhoff subcategories of $\Str(\mathbb L)$ are precisely classes given by extended $\kappa$-ary equational $\mathbb L$-theories.
\end{theo}
\begin{proof}
	One direction is again given by Proposition~\ref{Mod->Birk}. For the other, consider an enriched $\kappa$-Birkhoff subcategory $R\colon \cl\hookrightarrow\Str(\mathbb L)$. Arguing as in the proof of Theorem~\ref{birkhoff} above, we know that $\cl$ is given by the full subcategory of $\Str(\mathbb L)$ spanned by the objects orthogonal with respect to the reflections $\rho_Z\colon FZ\to F'Z$. Now, since $R$ preserves $\kappa$-filtered colimits then also $F'$ does; therefore each $\rho_Z$ is the $\kappa$-filtered colimit in $\Str(\mathbb L)^\to$ of the $\rho_Y$ with $Y\in \cv_\kappa\downarrow Z$. It follows that orthogonality with respect to $\rho_Y$, for $Y\in\cv_\kappa$, implies orthogonality with respect to each $\rho_Z$, for $Z\in\cv$. 
	
	Therefore $\cl$ is defined by the orthogonality class with respect to the maps $\rho_Z$ for any $Z\in\cv_\kappa$. Since these are regular epimorphisms we can conclude again thanks to Proposition~\ref{regular-epi-orth}.
\end{proof}

In Appendix~\ref{more-birk} we shall give hypotheses on $\cv$ and $\mathbb L$ so that the hypotheses of the theorem above are satisfied.

In certain situations, closure under $\kappa$-directed colimits can be replaced by closure under specific quotient maps:

\begin{defi}[\cite{AR2}]
	A morphism $f\colon A\to B$ in $\cv$ is called a $\kappa$-\textit{pure epimorphism} if it is projective with respect to the $\kappa$-presentable objects. Explicitly, if for every $\kappa$-presentable object $X$, all morphisms $X\to B$ factor through $f$.
\end{defi}

Every split epimorphism is $\lambda$-pure for every $\lambda$. In a locally $\lambda$-presentable category $\lambda$-pure morphisms are precisely the $\lambda$-filtered colimits in $\ck^\to$ of split epimorphisms in $\ck$ (see \cite[Proposition 3]{AR2}). In particular, they are epimorphisms (as the name suggests). 

\begin{lemma}
	Let $\mathbb L$ be a $\lambda$-ary language and $\lambda\leq\kappa\leq\infty$. Then every enriched $\kappa$-Birkhoff subcategory of $\Str(\mathbb L)$ is closed under quotients $f\colon A\to B$ for which $Uf$ is $\kappa$-pure.
\end{lemma} 
\begin{proof}
	Let $\cl$ be an enriched $\kappa$-Birkhoff subcategory of $\Str(\mathbb L)$. Following Theorem~\ref{l-birkhoff1}, $\cl$ is given by a $\kappa$-ary
	equational theory $\mathbb E$. Let $(s=t)$ be from $\mathbb E$, so that
	$s,t\colon FY\to FX$ where $X$ and $Y$ are $\kappa$-presentable. Consider
	$f\colon A\to B$ where $A\in\cl$. If $Uf$ is $\kappa$-pure, every 
	$g\colon FX\to B$ factors through $f$. Consequently, $B$ satisfies the equation $(s=t)$ in the unenriched sense. Since $\cl$ is closed under powers the satisfaction is actually enriched (Remark~\ref{unenrich-sat}); thus $B\in \cl$.
\end{proof}

\begin{rem}
	A Birkhoff subcategory closed under quotients $f\colon A\to B$ such that $Uf$
	is $\kappa$-pure does not need to be a $\kappa$-Birkhoff subcategory.
	An example for $\cv=\Set^\mathbb N$ and $\kappa=\omega$ is given in \cite{ARV}. Since epimorphisms in $\Set^\mathbb N$ split, every Birkhoff subcategory is $\omega$-Birkhoff.
\end{rem}

Recall that an initial object $0$ of $\cv$ is called \textit{strict initial}
if every morphism to $0$ is an isomorphism. We say that a category $\ck$ is \textit{strongly connected} if for every pair of objects $K$ and $K'$ of $\ck$, where $K'$ is not strict initial, there is a morphism $K\to K'$. (A slightly different notion, with initial rather than strict initial objects, was considered in~\cite{Bu}.)

\begin{propo}
	Let $\mathbb L$ be a $\lambda$-ary language, $\lambda\leq\kappa\leq\infty$, and $\cv_0$ be strongly connected. 
	Then enriched $\kappa$-Birkhoff subcategories of $\Str(\mathbb L)$ are precisely classes closed under products, powers, substructures, and quotients $f\colon A\to B$ for which $Uf$ is $\kappa$-pure.
\end{propo}
\begin{proof}
	One direction is given by the lemma above; for the other it is enough to show closure under $\kappa$-directed colimits. Let $(k_{ij}\colon K_i\to K_j)_{i\leq j\in M}$ be a $\kappa$-directed diagram in $\Str(\mathbb L)$ where $K_i\in\cl$ for all $i\in M$. Let $(k_i\colon K_i\to K)_{i\in M}$ be its colimit in $\Str(\mathbb L)$.
	We can assume that no $K_i$ is strictly initial. We will find a subalgebra $A$ of the product $\prod_{i\in M}K_i$ such that $K$ is a quotient $f\colon A\to K$ such that $Uf$ is $\lambda$-pure. This will be enough, given the closure properties.
	
	Consider products $\prod_{i\in M}K_i$ and $\prod_{i\geq j}K_i$ with projections $p_i\colon \prod_{i\in M}K_i\to K_i$ and
	$p^j_i\colon \prod_{i\geq j}K_i\to K_i$. Let $q_j\colon \prod_{i\in M}K_i\to\prod_{i\geq j}K_i$ be the projection, i.e., $p^j_iq_j=p_i$.
	Since $\cv_0$ is strongly connected, there are morphisms $u_{ji}\colon K_j\to K_i$ for $i< j$. Let $s_j\colon \prod_{i\geq j}K_i\to\prod_{i\in M}K_i$ be such
	that $p_is_j=p^j_i$ for $i\geq j$ and $p_is_j=u_{ji}p^j_j$ for $i<j$.
	Clearly $q_js_j=\id$, and hence $q_j$ is a split epimorphism.
	
	Let $u_j\colon K_j\to\prod_{i\geq j}K_i$ such that $p^j_iu_j=k_{ji}$. Then
	$u_j$ is a (split) monomorphism. Consider the pullback
	\begin{center}
		\begin{tikzpicture}[baseline=(current  bounding  box.south), scale=2]
			
			\node (a0) at (0,0.9) {$A_j$};
			\node (b0) at (1.1,0.9) {$\prod\limits_{i\in M}K_i$};
			\node (c0) at (0,0) {$K_j$};
			\node (d0) at (1.1,0) {$\prod\limits_{i\geq j}K_i$};
			
			\path[font=\scriptsize]
			
			(a0) edge [->] node [above] {$v_j$} (b0)
			(a0) edge [->] node [left] {$f_j$} (c0)
			(b0) edge [->] node [right] {$q_j$} (d0)
			(c0) edge [->] node [below] {$u_j$} (d0);
		\end{tikzpicture}	
	\end{center} 
	Then $v_j$ is a monomorphism and $f_j$ is a split epimorphism.	
	
	For $j'>j$, we get a morphism $a_{jj'}\colon A_j\to A_{j'}$ such that
	$f_{j'}a_{jj'}=k_{jj'}f_j$ and $v_{j'}a_{jj'}=q_{jj'}v_j$ where
	$q_{jj'}\colon \prod_{i\geq j}K_i\to\prod_{i\geq j'}K_i$ is the projection.
	Then $a_{jj'}$ form a directed diagram and we can take its colimit 
	$a_j\colon A_j\to A$. The induced morphism $v\colon A\to\prod_{i\in m}K_i$ is a monomorphism (by \cite[1.59]{AR}) and the induced morphism $Uf\colon UA\to UK$
	is a $\lambda$-pure epimorphism (by \cite[Proposition 3]{AR2}).
\end{proof}

See also Theorem~\ref{birkhoff4} where we consider closure under {\em $\ce$-quotients}, where $\ce$ is the left class of a factorization system on $\cv$.

\begin{rem}
The replacement of filtered colimits by pure quotients
is also consi\-de\-red in the recent paper \cite{Ka}.
\end{rem}

\section{Multi-sorted languages and theories}\label{multi-sort}

As classical single-sorted universal algebra has its multi-sorted version~\cite{BL}, so does our enriched theory. In this section we define multi-sorted languages, structures, terms, and equational theories. We then prove Theorem~\ref{multi-char} extending the ordinary results of~\cite[3.A]{AR}.

\begin{defi}
	A {\em multi-sorted language} $\mathbb L$ (over $\cv$) is the data of a set $S$ of sorts and a set of function symbols of the form
	$$f\colon((X_t)_{t\in T};(Y_u)_{u\in U})$$ 
	where $T,U\subseteq S$ and the arities $X_t$ and $Y_u$ are objects of $\cv$.
	The language $\mathbb L$ is called {\em $\lambda$-ary} if all the arities appearing in $\mathbb L$ lie in $\cv_\lambda$ and each $T,U$ above is of cardinality less than $\lambda$.
\end{defi}

We introduce the notion of $\mathbb L$-structure.

\begin{defi}
	Given a multi-sorted language $\mathbb L$, an {\em $\mathbb L$-structure} is the data of a family $A:=(A_s)_{s\in S}$ of objects in $\cv$ together with a morphism $$f_A\colon \textstyle\prod\limits_{t\in T} A_t^{X_t}\to \textstyle\prod\limits_{u\in U} A_u^{Y_u}$$ in $\cv$ for any function symbol $f\colon((X_t)_{t\in T};(Y_u)_{u\in U})$ in $\mathbb L$.
	
	A {\em morphism of $\mathbb L$-structures} $h\colon A\to B$ is determined by a family of morphisms $(h_s\colon A_s\to B_s)_{s\in S}$ in $\cv$ making the following square commute
	\begin{center}
		\begin{tikzpicture}[baseline=(current  bounding  box.south), scale=2]
			
			\node (a0) at (0,1) {$\prod\limits_{t\in T} A_t^{X_t}$};
			\node (b0) at (1.9,1) {$\prod\limits_{t\in T} B_t^{X_t}$};
			\node (c0) at (0,0) {$\prod\limits_{u\in U} A_u^{Y_u}$};
			\node (d0) at (1.9,0) {$\prod\limits_{u\in U} B_u^{Y_u}$};
			
			\path[font=\scriptsize]
			
			(a0) edge [->] node [above] {$\prod_{t\in T} h_t^{X_t}$} (b0)
			(a0) edge [->] node [left] {$f_A$} (c0)
			(b0) edge [->] node [right] {$f_B$} (d0)
			(c0) edge [->] node [below] {$\prod_{u\in U} h_u^{Y_u}$} (d0);
		\end{tikzpicture}	
	\end{center} 
	for any $f$ in $\mathbb L$.
\end{defi}

Since to give a map $\textstyle\prod\limits_{t\in T} A_t^{X_t}\to \textstyle\prod\limits_{u\in U} A_u^{Y_u}$ is the same as to give $\textstyle\prod\limits_{t\in T} A_t^{X_t}\to A_u^{Y_u}$ for each $u\in U$, in the definition above it would be enough to consider function symbols with a single output arity. Thus every multi-sorted language $\mathbb L$ can be replace by an $\mathbb L'$ where all the output arities are singletons; $\mathbb L$-terms and $\mathbb L'$-terms (as introduced below) will coincide thanks to the rule allowing tuples.

We keep this apparently more general approach since it will make it easier to write down equations in Theorem~\ref{multi-char}.

\begin{rem}\label{S_lambda}
	Given any set $S$, the $\cv$-category $\cv^S:=\prod_{s\in S}\cv$ is locally $\lambda$-presentable where the full subcategory $(\cv^S)_\lambda$ of the $\lambda$-presentable objects is spanned by those families $(X_s)_{s\in S}$ for which each $X_s$ is in $\cv_\lambda$ and $X_s\neq 0$ only for less than $\lambda$ indices. Equivalently, we identify the objects of $(\cv^S)_\lambda$ with families $(X_s)_{s\in S'}$ where the $X_s$ are $\lambda$-presentable in $\cv$ and $S'\subseteq S$ is $\lambda$-small.
\end{rem}

As in the single-sorted case we can build the $\cv$-category of $\mathbb L$-structures. Consider now the ordinary category $\cc(\mathbb L)^\lambda$ which has as objects the same of $(\cv^S)_\lambda$ (with the identification explained in Remark~\ref{S_lambda} above) and whose morphisms are freely generated under composition by the function symbols of $\mathbb L$, so that $f\colon((X_t)_{t\in T};(Y_u)_{u\in U})$ in $\mathbb L$ has domain $(X_t)_{t\in T}$ and codomain $(X_u)_{u\in U}$ in $\cc(\mathbb L)^\lambda$. Let now $\cc(\mathbb L)_\cv^\lambda$ be the free $\cv$-category on $\cc(\mathbb L)^\lambda$; then we take the pushout in $\cv\tx{-}\bo{Cat}$

\begin{center}
	\begin{tikzpicture}[baseline=(current  bounding  box.south), scale=2]
		
		\node (a0) at (0,0.8) {$|\ce|$};
		\node (b0) at (1.1,0.8) {$\cc(\mathbb L)_\cv^\lambda$};
		\node (c0) at (0,0) {$\ce^{\op}$};
		\node (d0') at (0.92,0.15) {$\ulcorner$};
		\node (d0) at (1.1,0) {$\Theta_\mathbb L^\lambda$};
		
		\path[font=\scriptsize]
		
		(a0) edge [->] node [above] {$j$} (b0)
		(a0) edge [->] node [left] {$i$} (c0)
		(b0) edge [->] node [right] {$H_\mathbb L$} (d0)
		(c0) edge [->] node [below] {$\theta_\mathbb L^\lambda$} (d0);
	\end{tikzpicture}	
\end{center} 
where $|(\cv^S)_\lambda|$ is the free $\cv$-category on the set of objects of $(\cv^S)_\lambda$, and $i$ and $j$ are the identity on objects inclusions. It follows that $H_\mathbb L$ and $\theta_\mathbb L^\lambda$ are the identity on objects as well.

\begin{defi}\label{cat-multi-L-struct}
	The $\cv$-category $\Str(\mathbb L)$ on a $\lambda$-ary multi-sorted language $\mathbb L$ is defined as the pullback
	\begin{center}
		\begin{tikzpicture}[baseline=(current  bounding  box.south), scale=2]
			
			\node (a0) at (0,0.8) {$\Str(\mathbb L)$};
			\node (a0') at (0.3,0.6) {$\lrcorner$};
			\node (b0) at (1.5,0.8) {$[\Theta_\mathbb L^\lambda,\cv]$};
			\node (c0) at (0,0) {$\cv^S$};
			\node (d0) at (1.5,0) {$[((\cv^S)_\lambda)^{\op},\cv]$};
			
			\path[font=\scriptsize]
			
			(a0) edge [right hook->] node [above] {} (b0)
			(a0) edge [->] node [left] {$U_\mathbb L$} (c0)
			(b0) edge [->] node [right] {$[\theta_\mathbb L^\lambda,\cv]$} (d0)
			(c0) edge [right hook->] node [below] {$\cv^S(K,1)$} (d0);
		\end{tikzpicture}	
	\end{center} 
	where $K\colon(\cv^S)_\lambda\hookrightarrow\cv^S$ is the inclusion.
\end{defi}
 
As for the single-sorted case, $\Str(\mathbb L)$ does not depend on the choice of $\lambda$.
 
Terms are constructed recursively as in Definition~\ref{terms} starting from the morphisms of $\cv^S$ and the function symbols of the language, and then closing under powers and superposition. Similarly we define interpretation of terms and multi-sorted equational theories.

\begin{rem}\label{multi-products}
	In the single-sorted case we saw that in Theorem~\ref{char-single} the $\cv$-categories of models of $\lambda$-ary equational theories can be characterized as those of the form $\lambda\tx{-Pw}(\ct,\cv)$ for a $\cv_\lambda$-theory $\ct$. And in Remark~\ref{products} we saw that preservation of $\lambda$-small products was implied by that of $\lambda$-small powers.
	
	This changes in the multi-sorted case, see Theorem~\ref{multi-char} below. The difference now is that the $\cv$-functor $\cv^S(K,1)\colon \cv^S\to [(\cv^S)_\lambda^{\op},\cv]$ restricts to an equivalence
	$$\cv^S\xrightarrow{\ \simeq\ }\lambda\tx{-PP} ((\cv^S)_\lambda^{\op},\cv)$$
	where on the right we consider those $\cv$-functor preserving $\lambda$-small powers and $\lambda$-small products. Here, preservation of products is necessary to obtain an inverse: this sends $F\colon ((\cv^S)_\lambda)^{\op}\to\cv$ to the family $(FI_s)_{s\in S}$, where $I_s\in(\cv^S)_\lambda$ is the $S$-tuple defined by the unit $I$ at $s$ and $0$ everywhere else.
\end{rem}

Below, we say that a parallel pair of maps $(f,g)$ in $\ck$ is {\em $\hat \cg$-split} if the pair of maps $(\ck(G,f),\ck(G,g))$ is split in $\cv$ for any $G\in\cg$. A set of objects $\cg\subseteq\ck$ is called {\em $\hat\cg$-projective} if for any $G\in\cg$ the $\cv$-functor $\ck(G,-)$ preserves coequalizers of $\hat \cg$-split pairs.

\begin{theo}\label{multi-char}
	The following are equivalent for a $\cv$-category $\ck$: \begin{enumerate}
		\item $\ck\simeq\Mod(\mathbb E)$ for a $\lambda$-ary multi-sorted equational theory $\mathbb E$;
		\item $\ck\simeq\tx{Alg}(T)$ for a $\lambda$-ary monad $T$ on $\cv^S$, for some set $S$;
		\item $\ck$ is cocomplete and has a strong generator $\cg\subseteq\ck$ made of  $\lambda$-presentable and $\hat\cg$-projective objects;
		\item $\ck\simeq\lambda\tx{-PP}(\cc,\cv)$ is equivalent to the full subcategory of $[\cc,\cv]$ spanned by those $\cv$-functors preserving $\lambda$-small products and $\lambda$-small powers, for some small $\cc$ with such limits.
	\end{enumerate}
\end{theo}
\begin{proof}
	$(3)\Rightarrow(2)$. Note that the $\cv$-category $\ck$ is locally $\lambda$-presentable and that $$U_\ck:=\textstyle\prod\limits_{G\in\cg}\ck(G,-)\colon\ck\longrightarrow\cv^S,$$ 
	where $S=\tx{Ob}(\cg)$, is (by hypothesis) continuous, $\lambda$-ary, and preserves coequalizers of $U_\ck$-split pairs. Thus $U_\ck$ has a left adjoint and is $\lambda$-ary monadic by the monadicity theorem.
	
	$(2)\Rightarrow(3)$. Conversely, now we have a monadic and $\lambda$-ary $\cv$-functor $U\colon\ck\to\cv^S$, with left adjoint $F$. Taking the values of $F$ at the singletons $I_s$, given as in Remark~\ref{multi-products}, we obtain a family $\{G_s\}_{s\in S}$ of objects of $\ck$ for which $U\cong\prod_{s\in S}\ck(G_s,-)$. Thus $\{G_s\}_{s\in S}$ satisfies the required properties.
	
	$(1)\Rightarrow(2)$. This is proved in the same manner as Proposition~\ref{eqations->monad}; just taking multiple copies of $\cv$.
	
	$(2)\Rightarrow(4)$. This is essentially a consequence of \cite{BG}. Consider the $(\cv^S)_\lambda$-theory $H\colon(\cv^S)_\lambda\to\ct$ corresponding to $T$; note that $H$ preserves all $\lambda$-small limits. Then by \cite[Theorem~19]{BG} the $\cv$-category $\tx{Alg}(T)$ is given by a pullback as in Definition~\ref{cat-multi-L-struct} with $\ct$ instead of $\Theta_\mathbb L^\lambda$. Now, by Remark~\ref{multi-products}, we have an equivalence 
	$$\cv^S(K,1)\colon\cv^S\xrightarrow{\ \simeq\ } \lambda\tx{-Pw}((\cv^S)_\lambda^{\op},\cv).$$ 
	Thus, since $\lambda\tx{-PP}(\ct,\cv)$ is the pullback of the inclusion $\lambda\tx{-PP}((\cv^S)_\lambda^{\op},\cv)\hookrightarrow[(\cv^S)_\lambda^{\op},\cv]$ along $[H,\cv]$; then it follows that also $ \ck\simeq\tx{Alg}(T)\simeq \lambda\tx{-PP}(\ct,\cv), $
	concluding the implication.
	
	$(4)\Rightarrow(1)$. Given $\cc$, let $S:=\tx{Ob}(\cc)$ be its set of objects; first we modify $\cc$ into a $(\cv^S)_\lambda$-theory, and then we argue as in Proposition~\ref{monad->equations}. 
	
	Consider the $\cv$-functor $F\colon((\cv^S)_\lambda)^{\op}\to\cc$ obtained by sending $(X_C)_{C\in S'}$, with $S'\subseteq S$ $\lambda$-small, to $\prod_{C\in S'}C^{X_C}$ in $\cc$. This is essentially surjective on objects (considering the images of the singletons on the unit), and we can take its (identity on objects, fully faithful) factorization. We then obtain an equivalence $E\colon \ct\to\cc$ (since $F$ was essentially surjective) and an identity on objects map $H\colon(\cv^S)_\lambda^{\op}\to\ct$ whose composite is $F$. It follows in particular that $H$ preserves $\lambda$-small products and $\lambda$-small powers (since $F$ does) and that $\ck\simeq\lambda\tx{-PP}(\ct,\cv)$; this implies that $\ck$ fits into the bipullback below (using the equivalence in Remark~\ref{multi-products}).
	\begin{center}
		\begin{tikzpicture}[baseline=(current  bounding  box.south), scale=2]
			
			\node (a0) at (0,0.8) {$\ck$};
			\node (a0') at (0.3,0.6) {$\lrcorner$};
			\node (b0) at (1.6,0.8) {$\lambda\tx{-Pw}(\ct,\cv)$};
			\node (c0) at (0,0) {$\cv^S$};
			\node (d0) at (1.6,0) {$\lambda\tx{-Pw}((\cv^S)_\lambda^{\op},\cv)$};
			
			\path[font=\scriptsize]
			
			(a0) edge [right hook->] node [above] {} (b0)
			(a0) edge [->] node [left] {} (c0)
			(b0) edge [->] node [right] {$[H,\cv]$} (d0)
			(c0) edge [right hook->] node [below] {} (d0);
		\end{tikzpicture}	
	\end{center} 
	Finally, it is now enough to argue as in the proof of Proposition~\ref{monad->equations} to show that $\ck$ can be described as follows: consider the multi-sorted language $\mathbb L$ with a function symbol $\overline f\colon((X_t)_{t\in T};(Y_u)_{u\in U})$ for any morphism $f$ of domain and codomain in $\ct$ equal to the input and output arities of $\overline f$. Note that for any map $g$ in $(\cv^S)_\lambda$ we have two different terms of the same arity given by $\theta_\mathbb L g$ and $\overline{Hg}$.
	The $\mathbb L$-theory $\mathbb E$ is given by the following equations:\begin{enumerate}
		\item[(a)] $\overline f( \overline g)=\overline{fg}$, for any composable maps $f,g$ in $\ct$;
		\item[(b)] $\overline{1_X} (\overline g)=\overline g$ and $\overline g( \overline{1_Y})=\overline g$, for any $g\colon X\to Y$ in $\ct$;
		\item[(c)] $\theta_\mathbb Lg=\overline{H(g)}$, for any morphism $g$ in $(\cv^S)_\lambda$;
		\item[(d)] $\overline{Z\otimes f} = \overline f^Z$ for any morphism $f\colon Y\to X$ in $\ct$ and $Z\in\cv_\lambda$; here $Z\otimes f\colon X\otimes Y\to Z\otimes X$ is the copower of $f$ by $Z$ in $\ct$.
	\end{enumerate}
	It is routine now to check that $\ck\simeq \Mod(\mathbb E)$.
\end{proof}

\appendix

\section{More on Birkhoff subcategories}\label{more-birk}
Here, we prove that the hypotheses of Theorem~\ref{birkhoff} and~\ref{l-birkhoff1} hold under some specific assumptions on $\cv$ and the arities of our languages.  

The condition that every strong epimorphism in $\cv$ is regular is certainly satisfied when $\cv_0$ is a regular category; that is, when regular epimorphisms in $\cv$ are stable under pullbacks. However, such a condition holds in a more general context. Below we give a sufficient condition for it to hold:

\begin{lemma}\label{strong=regular}
	Let $\ck$ be a ordinary category for which pullbacks of finite products of regular epimorphisms are epimorphisms. Then every strong epimorphism in $\ck$ is regular.
\end{lemma}
\begin{proof}
	Let $t\colon A\to B$ be a strong epimorphism in $\ck$. Consider the kernel pair $(h,k)\colon K\to A$ of $t$ and its coequalizer $q\colon A\to C$. So that there is an induced $s\colon C\to B$ such that $t=sq$, this is in particular a strong epimorphism (by the cancellativity property). To conclude it is enough to prove that $s$ is a monomorphism, since then it must be an isomorphism. For that consider any pair of maps $f,g\colon W\to C$ with $sf=sg$; then we can build the pullback of $q\times q\colon A\times A\to C\times C$ along $(f,g)\colon W\to C\times C$; this gives maps $f',g'$, and $e$ as in the solid diagram below.
	\begin{center}
		\begin{tikzpicture}[baseline=(current  bounding  box.south), scale=2]

			\node (a) at (0,0) {$A$};
			\node (a1) at (1,0) {$C$};
			\node (a2) at (2,0) {$B$};
			\node (b) at (0,-0.8) {$W'$};
			\node (b1) at (1,-0.8) {$W$};
			\node (z) at (-1,0) {$K$};
			
			\path[font=\scriptsize]
			
			(a) edge [->>] node [above] {$q$} (a1)
			(a1) edge [->>] node [above] {$s$} (a2)
			(b) edge [->>] node [below] {$q'$} (b1)
			
			([yshift=-1.5pt]z.east) edge [->] node [below] {$h$} ([yshift=-1.5pt]a.west)
			([yshift=1.5pt]z.east) edge [->] node [above] {$k$} ([yshift=1.5pt]a.west)
			([xshift=-1.5pt]b.north) edge [->] node [left] {$f'$} ([xshift=-1.5pt]a.south)
			([xshift=1.5pt]b.north) edge [->] node [right] {$g'$} ([xshift=1.5pt]a.south)
			([xshift=-1.5pt]b1.north) edge [->] node [left] {$f$} ([xshift=-1.5pt]a1.south)
			([xshift=1.5pt]b1.north) edge [->] node [right] {$g$} ([xshift=1.5pt]a1.south);
		\end{tikzpicture}
	\end{center}
	Here $q'$ is an epimorphism by hypothesis on $\ck$. By construction we have that $tf'=tg'$; thus there exists $p\colon W'\to K$ with $ht=f'$ and $kt=g'$. Since $q$ coequalizes $h$ and $k$, it follows that $qf'=qg'$, and hence $fq'=gq'$. But $q'$ is an epimorphism; thus $f=g$ and $s$ is a monomorphism.
\end{proof}

\begin{exams} Below we give a list of bases of enrichment for which strong and regular epimorphisms coincide.
	{\setlength{\leftmargini}{1.6em}
		\begin{enumerate}
			\item  If $\cv_0$ is a regular category then every strong epimorphism is regular. This includes examples such as the categories $\Set$ of sets, $\Ab$ of abelian groups, $\bo{GAb}$ of graded abelian groups, $\bo{DGAb}$ of differentially graded abelian groups, and $\bo{SSet}$ of simplicial sets.
			\item  $\cv=\Pos$ is the category of posets. Regular epimorphisms in $\Pos$ are in particular surjections, and hence satisfy the hypothesis of the lemma above; thus strong and regular epimorphisms in $\Pos$ coincide. Hence we obtain the Birkhoff-type theorems from \cite{Bl}.
			
			Strong epimorphisms in $\Pos$  are precisely those surjective maps $f\colon A\to B$ such that $y\leq y'$ in $B$ implies that there are 
			$x_0\leq x_1,x'_1\leq x_2, ..., x'_n\leq x_{n+1}$ such that $f(x_0)=y, f(x'_1)=f(x_1),\cdots, f(x_{n+1})=y'$ (see \cite[Theorem 3]{LN}).  
			Strong epimorphisms are closed under products but not stable under pullbacks because $\Pos$ is not regular.

			\item $\cv=\Met$ is the category of generalized metric spaces. As for $\Pos$, regular epimorphisms in $\Met$ are in particular surjections, and hence satisfy the hypothesis of the lemma above; thus strong and regular epimorphisms in $\Met$ coincide.
			
			Strong, and hence regular, epimorphisms in $\Met$ are precisely those surjective maps $f\colon A\to B$ such that $d(y,y')= \inf\{d(x,x')| fx=y, fx'=y'\}$ for any $y,y'$ in $B$. Clearly, every map with this property has the unique left lifting property to all monomorphisms, hence is a strong epimorphism. Conversely, it is easy to see that every morphism in $\Met$ factorizes as one satisfying the condition followed by a monomorphism. Applying this to a strong epimorphism $f$, it follows that the induced monomorphism is an isomorphism, and hence that $f$ satisfies the condition.
			
		\end{enumerate}
	}
\end{exams}

\begin{rem}
	If $\cv=\CMet$ is the full subcategory of $\Met$ consisting of complete metric spaces, then regular epimorphisms are not surjections; hence we cannot apply the lemma. It does not seem that strong epimorphisms are regular in $\CMet$.
	
	Similarly, in $\cv=\omega$-$\CPO$, regular epimorphisms are not surjections; hence we cannot apply the lemma. It does not seem that strong epimorphisms are regular in $\omega$-$\CPO$. Hence we do not obtain the Birkhoff-type theorem from \cite{ANR}.
\end{rem}

Next we recall the notion of $\ce$-projectivity~\cite{LR12} and introduce that of $\ce$-stability.

\begin{defi}
	Given a factorization system $(\ce,\cm)$ on $\cv$, recall that an object $X\in\cv$ is called {\em $\ce$-projective} if $\cv_0(X,-)\colon\cv\to\Set$ sends maps in $\ce$ to a surjection of sets. An object $X$ is called {\em $\ce$-stable} if $e^X\colon A^X\to B^X$ is in $\ce$ whenever $e\colon A\to B$ is in $\ce$. 
\end{defi}

If the unit is $\ce$-projective, then every $\ce$-stable object is $\ce$-projective. The following lemma gives sufficient conditions for the hypotheses of Theorem~\ref{birkhoff} to be satisfied.

\begin{lemma}\label{strong-pres}
	Consider an enriched factorization system $(\ce,\cm)$ on $\cv$ for which every map in $\cm$ is a monomorphism. Let $\mathbb L$ be a language whose function symbols have $\ce$-stable input arities. Then: \begin{enumerate}
		\item $(U^{-1}\ce,U^{-1}\cm)$ is a factorization system on $\Str(\mathbb L)$;
		\item $T:=UF\colon \cv\to\cv$ sends maps in $\ce$ to maps in $\ce$;
		\item if every strong epimorphism is regular in $\cv$ then the same holds in $\Str(\mathbb L)$; moreover $U$ preserves and reflects strong epimorphisms. 
	\end{enumerate} 
\end{lemma} 
\begin{proof}
	$(1)$. First let us prove that for any $g\colon A\to B$ in $\Str(\mathbb L)$, the $(\ce,\cm)$ factorization of the underlying morphism $Ug$ in $\cv$ lifts to a unique factorization $(e,m)$ of $g$ in $\Str(\mathbb L)$.
	
	Consider thus a morphism $g\colon A\to B$ in $\Str(\mathbb L)$, and the $(\ce,\cm)$ factorization $(e\colon A\to E, m\colon E\to B)$  of $Ug$ in $\cv$. We need to show that this induces a unique $\mathbb L$-structure on $E$ that makes $e$ and $m$ morphisms of $\mathbb L$-structures. For any $(X,Y)$-ary function symbol $f$ in $\mathbb L$, we can consider the induced diagram in $\cv$
	\begin{center}
		\begin{tikzpicture}[baseline=(current  bounding  box.south), scale=2]
			
			\node (a) at (-1,-0.8) {$A^Y$};
			\node (a0) at (0,-0.8) {$E^Y$};
			\node (b0) at (1,-0.8) {$B^Y$};
			
			\node (b) at (-1,0) {$A^X$};
			\node (c0) at (0,0) {$E^X$};
			\node (d0) at (1,0) {$B^X$};
			
			\path[font=\scriptsize]
			
			(a) edge [->] node [below] {$e^Y$} (a0)
			(a) edge [<-] node [left] {$f_A$} (b)
			(b) edge [->>] node [above] {$e^X$} (c0)
			
			(a0) edge [>->] node [below] {$m^Y$} (b0)
			(a0) edge [dashed, <-<] node [left] {$f_E$} (c0)
			(b0) edge [<-] node [right] {$f_B$} (d0)
			(c0) edge [>->] node [above] {$m^X$} (d0);
		\end{tikzpicture}	
	\end{center}
	where $e^X$ is still in $\ce$ since $X$ is $\ce$-stable, and $m^Y$ is still in $\cm$ since the factorization system is enriched. It follows that there is a unique arrow $f_E\colon E^X\to E^Y$ making the diagram above commute (by the orthogonality property of the factorization system). This endows $E$ with the desired $\mathbb L$-structure.
	
	This is enough to imply that $(U^{-1}\ce,U^{-1}\cm)$ is a factorization system in $\Str(\mathbb L)$; indeed, we know that the classes $U^{-1}\ce$ and $U^{-1}\cm$ are closed under composition and contain the isomorphisms, every morphisms $f$ in $\Str(\mathbb L)$ factors as $f=me$ with $e\in U^{-1}\ce$ and $m\in U^{-1}\cm$, and the factorization is functorial (since it is so in $\cv$, and $U$ is conservative).
	
	$(2)$. Given $e\colon A\to B$ in $\ce$, the map $UFe$ is in $\ce$ if and only if $Fe$ is in $U^{-1}\ce$, if and only if $Fe$ is left orthogonal with respect to $U^{-1}\cm$. Using that $F\dashv U$, it is easy to see that this last condition is equivalent to $e$ being orthogonal with respect to $\cm$, which is the case by hypothesis.
	
	$(3)$. Given any strong epimorphism $g\colon A\to B$ in $\Str(\mathbb L)$, we can consider the $(\ce,\cm)$ factorization of $Ug$ as above. Then, since this lifts to $g= me$ in $\Str(\mathbb L)$ where $m$ is a monomorphism ($U$ is conservative), it follows that $m$, and so also $Um$, is an isomorphism. Therefore $Ug$ coincides with its strong image, and is hence a strong epimorphism in $\cv$.
	
	To conclude, it remains to show that every strong epimorphism is regular in $\Str(\mathbb L)$. Given any such $g\colon A\to B$, we can consider its kernel pair $(h,k)\colon K\to A$; we ned to prove that $g$ is the coequalizer of $(h,k)$. For that, consider any other $\mathbb L$-structure $C$ and a map $t\colon A\to C$ with $th=tk$. By the arguments above $Ug$ is a (strong and hence) regular epimorphism; thus is the coequalizer of its kernel pair $(Uh,Uk)$. It follows that there exists a unique $s\colon UB\to UC$ for which $U(s\circ g)=Ut$. Using that the input arities of $\mathbb L$ are $\ce$-stable and that $Ug\in\ce$ (since maps in $\cm$ are monomorphisms, $\ce$ contains all strong epimorphisms), it is easy to see that $s$ is actually a map of $\mathbb L$-structures. This suffices to show that $g$ is a regular epimorphism. 
\end{proof}

Thus we obtain:

\begin{propo}\label{birkhoff2}
	Assume that in $\cv$ every strong epimorphism is regular and consider an enriched factorization system $(\ce,\cm)$ in $\cv$ for which every map in $\cm$ is a monomorphism. Let $\mathbb L$ be a language whose function symbols have $\ce$-stable input arities. Then, given $\lambda\leq\kappa\leq\infty$, enriched $\kappa$-Birkhoff subcategories of $\Str(\mathbb L)$ are precisely classes given by extended $\kappa$-ary equational $\mathbb L$-theories.
\end{propo}
\begin{proof}
	Follows from Theorem~\ref{l-birkhoff1} and Lemma~\ref{strong-pres}.
\end{proof}

We can apply this to the (strong epi, mono) = (regular epi, mono) factorization system in $\cv$ to obtain:

\begin{coro}
	Assume that in $\cv$ every strong epimorphism is regular. Let $\mathbb L$ be a language whose function symbols have regular epi-stable input arities. Then, given $\lambda\leq\kappa\leq\infty$, enriched $\kappa$-Birkhoff subcategories of $\Str(\mathbb L)$ are precisely classes given by extended $\kappa$-ary equational $\mathbb L$-theories.
\end{coro}

Next, we apply Proposition~\ref{birkhoff2} to a canonical factorisation system associated to $\cv$. A morphism $f\colon A\to B$ will be called a \textit{surjection} if $\cv_0(I,f)$ is surjective. For every object $Z$, the induced map $\delta_Z\colon Z_0\to Z$ is a surjection. Let $\Surj$ denote the class of all surjections in $\cv_0$ and let $\Inj$ be the class of morphisms of $\cv_0$ having the unique right lifting property with respect to every surjection. Morphisms from $\Inj$ will be called \textit{injections}.

\begin{rem}
	The paper \cite{R1} gives sufficient conditions for $(\Surj,\Inj)$ to be a factorization system on $\cv_0$. This will be an enriched factorization system whenever $\cv_0(I,-)$ is weakly strong monoidal in the sense of \cite{LT}; that is, if for any $X,Y\in\cv$ the induced map 
	$$\cv_0(I,X)\times \cv_0(I,Y)\to\cv_0(I, X\otimes Y)$$ 
	is surjective. Indeed, it is easy to see that in this case, if $f$ is a surjection then also $X\otimes f$ is one; therefore the factorization system is enriched by \cite[5.7]{LW}.
	Moreover, following \cite[3.4]{R1}, the injections are monomorphisms.
\end{rem}

Applying Proposition~\ref{birkhoff2} to this setting we immediately obtain the result below. Recall that a discrete object of $\cv$ is a coproduct of the unit.

\begin{coro}\label{birkhoff3}
	Assume that $(\Surj,\Inj)$ is a proper enriched factorization system in $\cv$, and let $\mathbb L$ be a $\lambda$-ary language whose function symbols have discrete input arities. Then, given $\lambda\leq\kappa\leq\infty$, enriched $\kappa$-Birkhoff subcategories of $\Str(\mathbb L)$ are precisely classes given by extended $\kappa$-ary equational $\mathbb L$-theories.
\end{coro}
\begin{proof}
	Since $(\Surj,\Inj)$ is proper, every surjection is an epimorphism and every regular epimorphism is a surjection. Moreover, surjections are stable under products and pullbacks since they are such in $\Set$; thus discrete objects are surjection-stable and every strong epimorphisms is regular in $\cv$ by Lemma~\ref{strong=regular}. Now the result follows by Proposition~\ref{birkhoff2}.
\end{proof}

We finish this section working within the hypotheses of Proposition~\ref{birkhoff2} and characterizing those enriched Birkhoff subcategories $\cl$ that are closed under {\em $\ce$-quotients}; meaning that, whenever $e\colon A\to B$ is such that $Ue\in\ce$ and $A\in\cl$, then also $B\in\cl$. 

\begin{theo}\label{birkhoff4}
	Assume that in $\cv$ every strong epimorphism is regular. Let $(\ce,\cm)$ be a proper enriched factorization system on $\cv$ such that for any $Y\in\cv_\kappa$ there exists an $\ce$-stable $X\in\cv_\kappa$ and a map $X\to Y$ in $\ce$.
	
	Consider a language $\mathbb L$ whose function symbols have $\ce$-stable input arities. Then enriched $\kappa$-Birkhoff subcategories of $\Str(\mathbb L)$ closed under $\ce$-quotients are precisely classes given by extended $\kappa$-ary equational $\mathbb L$-theories whose equations have $\ce$-stable arities. 
\end{theo}
\begin{proof}
	First observe that by Proposition~\ref{birkhoff2}, enriched $\kappa$-Birkhoff subcategories of $\Str(\mathbb L)$ coincide with those given by an $\kappa$-ary equational theory. 
	
	On one hand, if all equations in $\mathbb E$ have $\ce$-stable arities then $\Mod(\mathbb E)$ is closed under $\ce$-quotients. Indeed, consider a map $e\colon A\to B$ where $A$ is a $\mathbb E$-model and for which $Ue\in\ce$. Let $(s=t)$ be an equation from $\mathbb E$, with extended terms $s,t:(X,Y)$ where $X$ is $\ce$-stable by hypothesis (also $Y$ is, but that is not needed). We can consider the diagram below
	\begin{center}
		\begin{tikzpicture}[baseline=(current  bounding  box.south), scale=2]
			
			\node (z) at (-1,0) {$A^X$};
			\node (a) at (0,0) {$A^Y$};
			\node (b) at (-1,-0.8) {$B^X$};
			\node (b1) at (0,-0.8) {$B^Y$};

			\path[font=\scriptsize]
			
			(z) edge [->] node [left] {$e^X$} (b)
			(a) edge [->] node [right] {$e^Y$} (b1)
			
			([yshift=-1.5pt]z.east) edge [->] node [below] {$s_A$} ([yshift=-1.5pt]a.west)
			([yshift=1.5pt]z.east) edge [->] node [above] {$t_A$} ([yshift=1.5pt]a.west)
			([yshift=-1.5pt]b.east) edge [->] node [below] {$s_B$} ([yshift=-1.5pt]b1.west)
			([yshift=1.5pt]b.east) edge [->] node [above] {$t_B$} ([yshift=1.5pt]b1.west);
		\end{tikzpicture}
	\end{center}
	where $e^X$ is in $\ce$ since $X$ is $\ce$-stable. In particular $e^X$ is an epimorphism (by hypothesis on $\ce$); thus $s_A=t_A$ implies that $s_B=t_B$, and hence $B$ is a model of~$\mathbb E$.
	
	Conversely, let $\cl$ be an enriched $\kappa$-Birkhoff subcategory of $\Str(\mathbb L)$ closed under $\ce$-quotients. By the proof of Theorem~~\ref{l-birkhoff1} we know that $\cl$ is given by the orthogonality class defined with respect to the regular epimorphisms 
	$$ \rho_Z\colon FZ\longrightarrow F'Z $$
	where $Z\in \cv_\kappa$, and $\rho_Z$ is the reflection of $FZ$ into $\cl$. 
	
	Consider now the full subcategory $\cb$ of $\Str(\mathbb L)$ spanned by the objects orthogonal with respect to $\rho_X$ for any $\ce$-stable $X\in\cv_\kappa$. Then clearly we have inclusions $\cl\subseteq\cb\subseteq\Str(\mathbb L)$. To conclude, it is enough to prove that $\cl=\cb$, since $\cb$ has the desired description by the proof of Proposition~\ref{regular-epi-orth} --- use that, since $(\ce,\cm)$ is proper, the $\kappa$-presentable $\ce$-stable objects are a generating family.
	
	Applying the same arguments as above, since $\cb$ is also an enriched Birkhoff subcategory of $\Str(\mathbb L)$, we know that $\cb$ is reflective in $\Str(\mathbb L)$, that the reflections $\tau_Z\colon FZ\to F''Z$ are regular epimorphisms, and that orthogonality with respect to $\tau_Z$, for $Z\in\cv_\kappa$, in $\Str(\mathbb L)$ gives back $\cb$. This implies that for any $Z\in\cv_\kappa$ the map $\rho_Z$ factors as
	$$ \rho_Z\colon FZ\xrightarrow{\tau_Z} F''Z\xrightarrow{\eta_Z}F'Z$$
	where also $\eta_Z$ is an epimorphism (since both $\rho_Z$ and $\tau_Z$ were). To conclude, it is enough to prove that each $\eta_Z$, for $Z\in\cv_\kappa$, is an isomorphism. Indeed, $\cl$ and $\cb$ will then be described by the same orthogonality condition, making them the same subcategory of $\Str(\mathbb L)$. 
	
	Given any $\ce$-stable $X\in\cv_\kappa$, since $F''X\in\cb$ and by definition of $\cb$ as an orthogonality class, we obtain that the map $\eta_X$ is a split monomorphism, and hence an isomorphism. For a general $Z\in\cv_\kappa$, we can consider (by hypothesis on $\cv$) an $\ce$-stable $X\in\cv_\kappa$ together with $e\colon X\to Z$ in $U^{-1}\ce$. Thus in the diagram below
	\begin{center}
		\begin{tikzpicture}[baseline=(current  bounding  box.south), scale=2]
			
			\node (a0) at (0,0.8) {$FX$};
			\node (b0) at (1,0.8) {$F''X$};
			\node (e0) at (2,0.8) {$F'X$};
			\node (c0) at (0,0) {$FZ$};
			\node (d0) at (1,0) {$F''Z$};
			\node (f0) at (2,0) {$F'Z$};
			
			\path[font=\scriptsize]
			
			(a0) edge [->] node [above] {$\tau_{X}$} (b0)
			(b0) edge [->] node [above] {$\eta_{X}$} (e0)
			(b0) edge [->] node [below] {$\cong$} (e0)
			(a0) edge [->] node [left] {$Fe$} (c0)
			(b0) edge [->] node [left] {$F''e$} (d0)
			(e0) edge [->] node [right] {$F'e$} (f0)
			(c0) edge [->] node [below] {$\tau_Z$} (d0)
			(d0) edge [->] node [below] {$\eta_Z$} (f0);
		\end{tikzpicture}	
	\end{center} 
	we know that $Fe$ and $\tau_Z$ are in $U^{-1}\ce$ (the former by Lemma~\ref{strong-pres}, the latter being a regular epimorphism). Then, by \cite[2.1.4]{FK}, also $F''e$ is in $U^{-1}\ce$. But $F''X\in\cl$ and $\cl$ is closed under $\ce$-quotients; thus $F''Z\in\cl$ too. Therefore, $F''Z$ is orthogonal with respect to $\rho_Z$, making $\eta_Z$ a (split monomorphism and hence an) isomorphism.
\end{proof}

Finally, here are some examples:

\begin{exams}$ $
	{\setlength{\leftmargini}{1.6em}
		\begin{enumerate}
			\item If $\cv$ is a symmetric monoidal quasivariety~\cite{LT20} we can apply Theorem~\ref{birkhoff4}  to the (regular epi, mono) factorization system. Every regular projective object is $\ce$-stable; the converse holds if the unit $I$ is regular projective~\cite[Remark~4.15]{LT20}. 
			
			\item If $\cv=\Pos$ or $\Met$ we can apply Theorem~\ref{birkhoff4} to the $(\Surj,\Inj)$ factorization system, the $\Surj$-stable objects being the discrete ones.
			
			\item If in $\cv$ every strong epimorphism is regular, the unit $I$ is a generator, and epimorphisms are stable under products, then we can again apply Theorem~\ref{birkhoff4} to the (epi, strong mono) factorization system. Indeed, discrete objects are epi-stable and every object is covered, through an epimorphism, by a discrete object (since the unit is a generator). This applies in particular to $\cv=\Set,\Ab,\Pos$ and $\Met$.
			
		\end{enumerate}
	}
\end{exams}

\end{document}